\documentclass[11pt]{amsart}
\usepackage{amssymb,amsmath,amsthm,amsfonts,mathrsfs,graphicx}
\usepackage{mathtools}
\usepackage{subcaption}
\usepackage[margin=3cm]{geometry}
\usepackage[colorlinks]{hyperref}
\usepackage{tikz}

\title[The Euler-Poisson-alignment system]{On the Euler-Poisson equations with variable background states and nonlocal velocity alignment}

\author{Kunhui Luan}
\address[Kunhui Luan]{\newline Department of Mathematics, \ 
 University of South Carolina, 1523 Greene St., Columbia, SC 29208, USA}
\email{kunhui.luan@sc.edu}

\author{Changhui Tan}
\address[Changhui Tan]{\newline Department of Mathematics, \ 
 University of South Carolina, 1523 Greene St., Columbia, SC 29208, USA}
\email{tan@math.sc.edu}

\author{Qiyu Wu}
\address[Qiyu Wu]{\newline Department of Mathematics, \ 
 University of South Carolina, 1523 Greene St., Columbia, SC 29208, USA}
\email{qiyu@email.sc.edu}

\thanks{\textit{Acknowledgment.} This work has been supported by the NSF grants DMS-2108264 and DMS-2238219.}

\makeatletter
\@namedef{subjclassname@2020}{\textup{2020} Mathematics Subject Classification}
\makeatother


\newtheorem{theorem}{Theorem}[section]
\newtheorem{lemma}[theorem]{Lemma}

\newtheorem{proposition}[theorem]{Proposition}

\theoremstyle{definition}

\theoremstyle{remark}
\newtheorem{remark}{Remark}[section]

\numberwithin{equation}{section}

\def\R{\mathbb{R}}
\def\T{\mathbb{T}}
\def\dd{\mathrm{d}}
\def\pa{\partial}

\def\psim{\psi_-}
\def\psiM{\psi_+}
\def\cm{c_-}
\def\cM{c_+}
\def\nvm{\nu_-}
\def\nvM{\nu_+}
\def\CTsub{\widehat\Sigma_\flat} 
\def\CTsup{\widehat\Sigma_\sharp}
\def\ctsub{\Sigma_\flat} 
\def\ctsup{\Sigma_\sharp} 
\def\L{\mathcal{L}}
\def\Lt{\widetilde{\L}}
\def\Rt{\widetilde{R}}
\def\Ct{\widetilde{C}}
\def\wt{\widetilde{w}}
\def\st{\widetilde{s}}
\def\Pt{\widetilde{P}}
\def\Nt{\widetilde{N}}
\def\tt{\widetilde{t}}

\newcommand*\circled[1]{\tikz[baseline=(char.base)]{
            \node[shape=circle,draw,inner sep=1pt] (char) {#1};}}

\begin{document}
\allowdisplaybreaks

\begin{abstract}
We study the 1D pressureless Euler-Poisson equations with variable background states and nonlocal velocity alignment. Our main focus is the phenomenon of critical thresholds, where subcritical initial data lead to global regularity, while supercritical data result in finite-time singularity formation. The critical threshold behavior of the Euler-Poisson-alignment (EPA) system has previously been investigated under two specific setups: 
(1) when the background state is constant, phase plane analysis was used in \cite{bhatnagar2023critical} to establish critical thresholds; and 
(2) when the nonlocal alignment is replaced by linear damping,  comparison principles based on Lyapunov functions were employed in \cite{choi2024critical}.

In this work, we present a comprehensive critical threshold analysis of the general EPA system, incorporating both nonlocal effects. Our framework unifies the techniques developed in the aforementioned studies and recovers their results under the respective limiting assumptions. A key feature of our approach is the oscillatory nature of the solution, which motivates a decomposition of the phase plane into four distinct regions. In each region, we implement tailored comparison principles to construct the critical thresholds piece by piece.
 \end{abstract}

\maketitle 

\tableofcontents

\section{Introduction}\label{sec:intro}
We consider the following one-dimensional pressureless Euler equation with nonlocal interacting forces:
\begin{equation}\label{eq:EPA}
\begin{cases}
 \,\,\pa_t \rho +\pa_x(\rho u)=0, \qquad x\in\Omega, \quad t\geq0,\\
 \,\,\pa_t (\rho u) +\pa_x (\rho u^2) = -k \rho\pa_x\phi+\displaystyle\int_\Omega \psi(x-y)\big(u(y)-u(x)\big)\rho(x)\rho(y)\,\dd y,\\
 \,\,-\pa_{x}^2\phi = \rho - c,
\end{cases}
\end{equation}
subject to initial data
\begin{equation}\label{eq:init}
  \rho(x,0) = \rho_0(x), \quad u(x,0) = u_0(x).
\end{equation}
Here, $(\rho,u)$ denote the density and velocity, respectively. The spatial domain $\Omega$ is taken to be the periodic domain $\T$, though it can be extended to the entire real line $\R$.

The system is known as the \emph{Euler-Poisson-alignment} (EPA) system. The right-hand side of the momentum equation \eqref{eq:EPA}$_2$ includes two types of interacting forces: the \emph{attractive-repulsive force} and the \emph{alignment force}.

\subsection{The Euler-Poisson equations}
When the alignment force is absent ($\psi \equiv 0$), the system \eqref{eq:EPA} reduces to the pressureless \emph{Euler-Poisson equations}, which model various physical phenomena, including semiconductor and plasma dynamics. The interaction force is represented by a potential $\phi$ generated by the electric field, which is governed by the Poisson equation \eqref{eq:EPA}$_3$. The parameter $k$ quantifies the strength of the electric force: it is attractive when $k < 0$ and repulsive when $k > 0$. The function $c=c(x,t)$ represents the background charge. 

The global well-posedness theory for the 1D pressureless Euler-Poisson equations was first established in \cite{engelberg2001critical} under the assumption of a constant background charge. The global behavior of the solution exhibits a \emph{critical threshold} phenomenon, where the initial data create a dichotomy: subcritical initial data lead to global smooth solutions, while supercritical data result in finite-time singularity formation. An explicit expression for the threshold (see \eqref{eq:CTEP}) was derived, providing a sharp distinction between these two regions.

A generalization of the Euler-Poisson equation arises when the background charge is not constant, referred to as the \emph{variable background} case. The critical threshold phenomenon for this system has been recently investigated in \cite{bhatnagar2020critical, choi2024critical, rozanova2025repulsive}. In particular, \cite{choi2024critical} establishes subcritical and supercritical regions under the presence of additional velocity damping. To address the challenges posed by the variable background, a series of \emph{comparison principles} were developed based on Lyapunov functions introduced in \cite{bhatnagar2020critical2}, leading to the identification of non-trivial subcritical and supercritical regions

We would like to mention other extensions, though these will not be the focus of this paper. First, the Euler-Poisson equations in two or three dimensions are less understood. Global theory for smooth solutions is known for radially symmetric data \cite{wei2012critical,tan2021eulerian,carrillo2023existence}. For general data, critical thresholds are established for a 2D restricted Euler-Poisson (REP) system in \cite{tadmor2003critical}, but global regularity for the Euler-Poisson equations remains an open question. Second, a series of works have explored weak solutions to the Euler-Poisson equations, particularly in 1D \cite{gangbo2009euler, natile2009wasserstein, brenier2013sticky, carrillo2023equivalence}, leveraging the system's geometric structure. Finally, the equations have also been studied with the inclusion of pressure \cite{guo1998smooth, tadmor2008global, guo2011global}.

\subsection{The Euler-alignment system}
When the attractive-repulsive force is absent ($k=0$), the system \eqref{eq:EPA} reduces to the pressureless \emph{Euler-alignment system} (EAS), which models flocking behavior in animal swarms. This system serves as the macroscopic representation of the Cucker-Smale dynamics \cite{cucker2007emergent}. The function $\psi$ is known as the \emph{communication protocol}, which quantifies the strength of pairwise alignment interactions. A natural assumption is that $\psi$ is a radially decreasing function, meaning that communication weakens as the distance between agents increases.

The global behavior of solutions to the EAS varies depending on the type of communication protocol. When $\psi$ is \emph{bounded}, critical thresholds were established in \cite{tadmor2014critical,carrillo2016critical, he2017global}. In particular, a sharp critical threshold for the 1D EAS was found in \cite{carrillo2016critical}. Similar behavior occurs when $\psi$ is \emph{weakly singular}, meaning it is unbounded but integrable at the origin \cite{tan2020euler,leslie2020lagrangian}. Another interesting scenario arises when $\psi$ is \emph{strongly singular}, meaning it is unintegrable at the origin. In this case, strong alignment leads to global regularity of the 1D EAS for all smooth initial data \cite{shvydkoy2017eulerian,do2018global}.

The multi-dimensional EAS has been studied under radial symmetry \cite{tan2021eulerian} and in the case of uni-directional flow \cite{lear2021unidirectional, lear2022existence, li2024global}. When pressure is present, the system has been analyzed in \cite{choi2019global, constantin2020entropy, chen2021global}. Recent studies on weak solutions to the 1D EAS can be found in \cite{leslie2018weak, leslie2023sticky, leslie2024finite, galtung2025sticky}.

\subsection{The Euler-Poisson-alignment system}
When both forces are active, the EPA system serves as a prototypical example of the hydrodynamic three-zone interaction model in collective dynamics, featuring long-range attraction, short-range repulsion, and mid-range alignment.

The critical threshold phenomenon for the 1D EPA system was first discussed in \cite{carrillo2016critical} for the case where the electric field has zero background ($c \equiv 0$) and the alignment force has a bounded communication protocol. In \cite{kiselev2018global}, the EPA system with a strongly singular communication protocol is investigated. It is shown that strong alignment interaction dominates the attractive-repulsive force, ensuring global regularity of the solution.

The 1D EPA system with a positive constant background charge and a bounded communication protocol was studied in a recent work \cite{bhatnagar2023critical}. Compared to the zero background charge case, the main challenge lies in the oscillatory nature of the solution over time. A common approach to handling the nonlocal alignment interaction is the use of comparison principles to relate the solution to a local system. However, the oscillations make it difficult to apply these principles directly. To overcome this, a careful phase-plane analysis was developed, providing a framework for alternating between two comparison principles, ultimately leading to the construction of critical thresholds. The result has also been extended to weakly singular communication protocols.

The main purpose of this paper is to study the critical threshold phenomenon for the EPA system \eqref{eq:EPA} under a general setup, incorporating two sources of nonlocal input: first, an electric force with a variable background charge $c = c(x,t)$, and second, a nonlocal alignment force with $\psi = \psi(r)$.

To effectively handle both nonlocal effects in the presence of oscillatory solutions, we extend the approach in \cite{bhatnagar2023critical} by constructing threshold regions using \emph{four} comparison principles. We partition the phase plane into four regions, assigning a specific comparison principle to each. This allows us to define subcritical and supercritical regions that adhere to these principles, forming \emph{invariant regions} that guarantee global well-posedness in the subcritical case and finite-time blowup in the supercritical case. See Theorem \ref{thm:thresholds} for the precise statement of our main result.

Moreover, we demonstrate that our critical thresholds can be explicitly formulated by utilizing the Lyapunov functions from \cite{choi2024critical}. In fact, when $\psi$ is taken as a constant, our constructed threshold regions coincide with those established in \cite{choi2024critical}.

Our result provides a unified approach that includes both \cite{bhatnagar2023critical} and \cite{choi2024critical} as special cases, where only one nonlocal input is turned on in each case. Our approach generalizes both cases by simultaneously incorporating both nonlocal effects, showing that the two analytical methods used in these previous works are equivalent within our broader framework.

\subsection{Outline of the paper}
The rest of the paper is organized as follows. In Section \ref{sec:results}, we outline the assumptions and present the main results for the EPA system \eqref{eq:EPA}. Section \ref{sec:construction} focuses on the construction of the critical thresholds, where we introduce Lyapunov functions to provide explicit expressions for the threshold regions. In Section \ref{sec:proof}, we establish comparison principles based on Lyapunov functions, and apply them to prove the main critical thresholds theorem. Section \ref{sec:phase} contains an auxiliary phase plane analysis, which provides explicit expressions for key parameters and the admissible conditions. Finally, we conclude with a discussion on various extensions in Section \ref{sec:extension}.

\section{Statements of main results}\label{sec:results}
In this section, we present our results on the Euler-Poisson-alignment system \eqref{eq:EPA} with variable background states.

We start with assumptions on the background state $c=c(x,t)$. The first assumption is a \emph{compatibility condition}
\begin{equation}\label{A1}\tag{A1}
 \int_\T c(x,t)\,\dd x = \int_\T \rho_0(x)\,\dd x.	
\end{equation}
This condition arises from the Poisson equation \eqref{eq:EPA}$_3$. Integrating the equation in $x$ gives
\begin{equation}\label{eq:compcond}
\int_\T\rho(x,t)-c(x,t)\,\dd x=0.	
\end{equation}
Using the conservation of mass, we obtain \eqref{A1}. Thus, the compatibility condition is necessary to ensure the solvability of the Poisson equation \eqref{eq:EPA}$_3$. More precisely, under \eqref{eq:compcond}, there exists a solution $\phi$ to \eqref{eq:EPA}$_3$, unique up to a constant shift. Consequently, the Poisson force $\pa_x\phi$ in \eqref{eq:EPA}$_2$ is uniquely determined by $\rho-c$.

Additionally, we impose the following regularity assumption on the background state:
\begin{equation}\label{A2}\tag{A2}
 c \in W^{1,\infty}\big(\R_+; H^{s}(\T)\big)~\cap~L^\infty\big(\R_+; H^{s+1}(\T)\big),	
\end{equation}
as well as the upper and the lower bounds
\begin{equation}\label{A3}\tag{A3}
  0 < \cm \leq c(x,t) \leq \cM,\quad \forall~ x\in\T,\,\,t\geq0.		
\end{equation}

We will focus on the case where the Poisson interaction is repulsive, namely
\begin{equation}\label{A4}\tag{A4}
  k>0.	
\end{equation}

Finally, we assume the communication protocol $\psi$ in the alignment force is bounded from above and below:
\begin{equation}\label{A5}\tag{A5}
  0 < \psim \leq \psi(x) \leq \psiM,\quad \forall~ x\in\T.	
\end{equation}

\subsection{Reformulation of the system}
Let us introduce an auxiliary quantity 
\begin{equation}\label{eq:G}
  G = \pa_x u + \psi\ast\rho,	
\end{equation}
where $\ast$ denotes the convolution in $x$-variable. This quantity was first introduced in \cite{carrillo2016critical} for the Euler-alignment system. One can easily deduce from \eqref{eq:EPA} that $G$ satisfies the following dynamics:
\[
 \pa_t G + \pa_x(Gu) = k(\rho-c).
\]
Therefore, the system \eqref{eq:EPA} can be equivalently reformulated as the dynamics of $(\rho, G)$ as follows:
\begin{equation}\label{eq:main}
  \begin{cases}
    \partial_t \rho +\partial_x(\rho u)=0,  \\
    \partial_t G+\partial_x(Gu) = k(\rho-c).
  \end{cases}
\end{equation}
The velocity $u$ can be recovered from the relation \eqref{eq:G}, up to a constant shift. Integrating \eqref{eq:EPA}$_2$ in $x$ yields the conservation of momentum:
\[\frac{\dd}{\dd t}\int_\T\rho u\,\dd x = 0,\quad \text{namely}\quad \int_\T\rho(x,t)u(x,t)\dd x = \int_\T\rho_0(x)u_0(x)\dd x,\]
which can be used to uniquely determine the constant shift.

\subsection{Local well-posedness}
We state the following local well-posedness and regularity criterion result for the EPA system.
\begin{theorem}[Local well-posedness]\label{thm:LWP} 
  Let $s>\frac{1}{2}$. Consider the system \eqref{eq:main} with initial data $(\rho_0,G_0)$ satisfying
    \begin{equation}\label{eq:reginit}
        \rho_0>0,\quad (\rho_0-c_0,G_0) \in  H^s(\T)\times H^s(\T).
    \end{equation}
  Assuming \eqref{A1}--\eqref{A5} hold.   
    Then there exists a positive constant $T>0$ such that solution
   \begin{equation}\label{eq:regrhoG}
        (\rho,G) \in C\big([0,T];H^s(\T)\big)\times C\big([0,T];H^s(\T)\big). 
    \end{equation}
  Consequently, the EPA system \eqref{eq:EPA} has a unique solution 
    \begin{equation}\label{eq:regrhou}
        (\rho,u) \in C\big([0,T];H^s(\T)\big)\times C\big([0,T];H^{s+1}(\T)\big). 
    \end{equation} 
  Moreover, $T$ can be extended as long as
    \begin{equation}\label{eq:BKM}
        \int_0^T \Big(\|\rho(\cdot,t)\|_{L^\infty}+\|G(\cdot,t)\|_{L^\infty}\Big)\,\dd t< \infty.
    \end{equation}   
\end{theorem}
The local well-posedness theory for the EPA system with constant background has been established in, e.g. \cite[Theorem 2.1]{bhatnagar2023critical}. The variable background is addressed in, e.g. \cite{choi2024critical}, without the alignment force. For the sake of completeness, we provide a proof of Theorem \ref{thm:LWP} in Appendix \ref{sec:LWP}. 

\subsection{Global regularity versus finite time blowup}
The criterion \eqref{eq:BKM} suggests that global regularity depends on the boundedness of $\rho$ and $G$. To investigate this, we analyze the dynamics of these quantities along each characteristic path $X(t)=X(t;x)$ originating from $x\in\T$, defined by
\[\frac{\dd}{\dd t}X(t) = u\big(X(t),t\big),\quad X(0) = x.\]
Denote $^\prime$ as the material derivative along the characteristic path
\[f'(t): = \frac{\dd}{\dd t}f(X(t),t) = \pa_t f(X(t),t) + u(X(t),t)\,\pa_xf(X(t),t)).\]
From the system \eqref{eq:main} and the relation \eqref{eq:G}, we derive the following coupled dynamics:
\begin{equation}\label{eq:rhoG}
 \begin{cases}
   \rho' = - \rho u = -\rho (G-\psi\ast\rho),\\
   G' = - Gu + k(\rho-c) = -G(G-\psi\ast\rho)+k(\rho-c).
 \end{cases}	
\end{equation}

The analysis of the dynamics \eqref{eq:rhoG} along characteristic paths reveals a \emph{critical threshold} phenomenon: the boundedness of $\rho$ and $G$ depends on the initial data. Specifically, subcritical initial data ensures that $(\rho, G)$ remain bounded for all time, leading to global well-posedness by Theorem \ref{thm:LWP}, as \eqref{eq:BKM} holds for all time. Conversely, supercritical initial data results in finite-time singularity formation.

\begin{theorem}[Critical thresholds]\label{thm:thresholds}
Consider the EPA system \eqref{eq:main} with the assumptions in Theorem \ref{thm:LWP}. The following results hold:
 \begin{itemize}
  \item Global well-posedness: under appropriate admissible conditions, there exists a subcritical region $\CTsub\subset\R\times\R_+$, such that if the initial data $(\rho_0, G_0)$ satisfy
   \[\big(G_0(x),\rho_0(x)\big)\in \CTsub,\quad \forall~x\in\T,\]
   then the EPA system \eqref{eq:main} has a unique global smooth solution, in the sense that \eqref{eq:regrhoG} holds for all time $T$.
  \item Finite time blowup: there exists a supercritical region $\CTsup\subset\R\times\R_+$, such that if the initial data $(\rho_0, G_0)$ satisfy
  \[\exists~x_*\in\T \quad\text{such that}\quad \big(G_0(x_*),\rho_0(x_*)\big)\in \CTsup,\]
   then smooth solution to the EPA system \eqref{eq:main} stops existing at a finite time $T_*$. Moreover, the solution generates a singular shock at the blowup time $T_*$ and location $X_* = X(T_*; x_*)$, namely
   \begin{equation}\label{eq:singularshock}
	\lim_{t\to T_*^-}\rho(X_*,t)=\infty,\quad \lim_{t\to T_*^-}G(X_*,t)=\lim_{t\to T_*^-}\pa_xu(X_*,t)=-\infty.     
   \end{equation}
 \end{itemize}
\end{theorem}

The critical threshold phenomenon was first studied in \cite{engelberg2001critical} for the Euler-Poisson equation with a constant background charge, which corresponds to the case where $\psi \equiv 0$ and $c$ is a constant. In this setting, the coupled dynamics \eqref{eq:rhoG} can be solved explicitly, yielding explicit subcritical and supercritical regions:
\begin{equation}\label{eq:CTEP}	
\CTsub = \left\{(G,\rho) : G^2 < \sqrt{k(2\rho-c)}\right\},\quad \text{and}\quad 
\CTsup = (\R\times\R_+)\backslash\CTsub.
\end{equation}
Since these two regions partition the entire space $\R\times\R_+$, the threshold condition is \emph{sharp}.

For the general EPA system \eqref{eq:main}, the threshold regions $\CTsub$ and $\CTsup$ in Theorem \ref{thm:thresholds} are expressed in a less explicit form. Figure \ref{fig:CT} illustrates the critical thresholds under three typical setups, as introduced in \cite{bhatnagar2023critical}: 
\begin{itemize}
 \item weak alignment: $0<\nvM<2\sqrt{k\cm}$,
 \item median alignment: $\nvM\geq 2\sqrt{k\cm}$ and $\nvm < 2\sqrt{k\cM}$,
 \item strong alignment: $\nvm\geq 2\sqrt{k\cM}$,
\end{itemize}
where $\nvM$ and $\nvm$ are defined in \eqref{eq:nupm}, measuring the strength of the alignment force.

We will provide detailed expressions for the threshold regions under each setup  in Theorems \ref{thm:sub} and \ref{thm:sup}. In particular, some \emph{admissible conditions} \eqref{AC1}-\eqref{AC3} are required to obtain $\CTsub$ in weak and median alignment cases.

\begin{figure}[h!]
\centering
\includegraphics{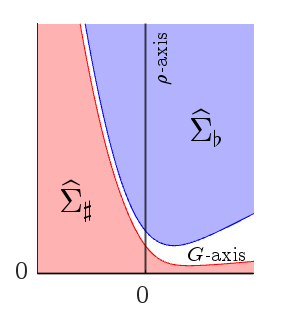}	
\includegraphics{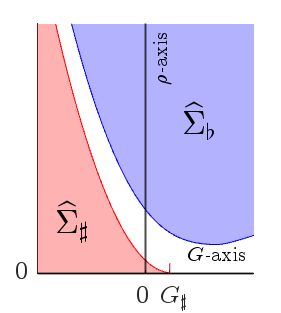}	
\includegraphics{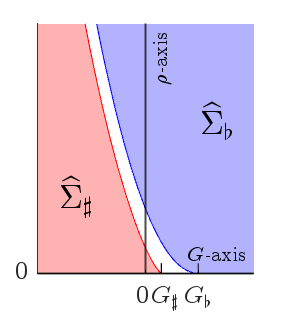}	
\caption{Illustration of the critical thresholds $\CTsub$ and $\CTsup$. Left: weak alignment; middle: median alignment; right: strong alignment.}\label{fig:CT}
\end{figure}

As seen in Figure \ref{fig:CT}, there is a gap between the two threshold conditions. So our threshold conditions are not necessarily sharp. This gap primarily arises from the nonlocal effects. In the special case where $\cm=\cM$ and $\nvm=\nvM$, our result becomes sharp. If we additionally assume $\nvm=\nvM=0$, the system reduces to the Euler-Poisson equation with a constant background charge, and our threshold conditions recover \eqref{eq:CTEP}. The points $G_\flat$ and $G_\sharp$ in the figure have the explicit form \eqref{eq:Gflat} and \eqref{eq:Gsharp}, respectively.

Our result is also compatible with the previous works. When $\cm=\cM$, the thresholds we derive coincide with the conditions obtained in \cite{bhatnagar2023critical}. When $\nvm=\nvM$, our thresholds match those in \cite{choi2024critical}. Notably, we establish a unified approach that aligns with both prior results, despite these results being derived through different analytical methods. 

\section{Construction of the critical thresholds}\label{sec:construction}
In this section, we construct the critical threshold regions $\CTsub$ and $\CTsup$ for the EPA system \eqref{eq:main}.

Compared with the Euler-Poisson equations with constant background, the introduction of a variable background and alignment force presents a significant challenge in analyzing critical thresholds. Specifically, the dynamics \eqref{eq:rhoG} depend on the nonlocal quantities:
\[ c = c(x,t), \quad \text{and}\quad \nu = \nu(x,t) :=\psi\ast\rho(x,t) = \int_\T\psi(x-y)\rho(y,t)\,\dd y,\] 
whose values change along the characteristic path and may not be determined solely by local information. 

The main idea is to develop \emph{comparison principles} that allow us to compare \eqref{eq:rhoG} with localized systems in which $c$ and $\nu$ are replaced by fixed values.

From the assumptions \eqref{A3} and \eqref{A5}, we have the following bounds on $\nu$ and $c$:
\begin{equation}\label{eq:nuc}
 0<\cm\leq c\leq \cM, \quad\text{and}\quad 0<\nvm\leq\nu\leq\nvM, 
\end{equation}
where
\begin{equation}\label{eq:nupm}
 \nvm = \psim \int_\T\rho_0(x)\,\dd x,\quad
 \nvM = \psiM \int_\T\rho_0(x)\,\dd x.
\end{equation}

Previous works have successfully addressed two key scenarios:
\begin{itemize}
\item Constant $c$ and variable $\nu$. This corresponds to the EPA system with a constant background. Critical thresholds were established in \cite{bhatnagar2023critical}, where the threshold regions were constructed piecewise using comparison principles applied to localized systems with $\nu$ replaced by $\nu_\pm$.
\item Variable $c$ and constant $\nu$. This setting corresponds to the Euler-Poisson equations with velocity damping. Critical thresholds were recently studied in \cite{choi2024critical}, where comparison principles were applied through a different approach based on a set of Lyapunov functions.
\end{itemize}

Our goal is to develop a universal approach that applies to both scenarios, as well as the general case where neither $\nu$ nor $c$ is constant. Additionally, we will demonstrate the equivalence of the two different  approaches on comparison principles.

\subsection{A reformed coupled system}
For analytical convenience, following \cite{bhatnagar2023critical}, we introduce two new variables:
\begin{equation}\label{eq:ws}
 w:=\frac{G}{\rho},\quad s:=\frac{1}{\rho}.
\end{equation}
The dynamics of $(w,s)$ can be deduced from \eqref{eq:rhoG} and takes the form:
\begin{equation}\label{S}\tag{S}
   \begin{cases}
    w' = k(1-cs),\\
    s' = w-\nu s.
   \end{cases}
\end{equation}

Note that the forcing term on the right-hand side of \eqref{S} is Lipschitz in $(w, s)$, ensuring that the solution $(w(t), s(t))$ remains bounded for all finite times. Consequently, the only possible finite-time blowup scenario occurs when  $s(T_*) = 0$, which corresponds to $\lim_{t \to T_*^-} \rho(t) = \infty$.

We will construct threshold regions in $\R\times\R_+$ for the $(w,s)$ dynamics \eqref{S} with the following properties:
\begin{itemize}
 \item Subcritical region $\ctsub$: if $(w_0,s_0)\in\ctsub$, then $s(t)>0$ for all time.
 \item Supercritical region $\ctsup$: if  $(w_0,s_0)\in\ctsup$, then there exists a finite time $T_*$ such that $s(T_*)=0$.
\end{itemize}

The threshold regions $\CTsub$ and $\CTsup$ on $(G,\rho)$ can then be determined using the relation \eqref{eq:ws}. This classification ultimately dictates whether the EPA system \eqref{eq:main} exhibits global regularity or finite-time singularity formation.

To effectively handle the nonlocality, we compare the system \eqref{S} with a localized system in which $c$ and $\nu$ are replaced by $c_\pm$ and $\nu_\pm$, respectively. The four corresponding localized systems are denoted by (\ref{S}$^{\pm\pm}$), where the first  $\pm$ refers to $c_\pm$ and the second to $\nu_\pm$.

We highlight a major challenge: the solutions $(w(t), s(t))$ exhibit oscillatory behavior over time. As a result, establishing comparison principles is highly non-trivial. To address this, we must alternate the comparison among all four localized systems, depending on the phase in which the solution resides at a given time.

\subsection{Subcritical region}
Next, we describe the construction of the subcritical region $\ctsub$, which is enclosed by up to four trajectories of the localized systems.

\begin{itemize}  
\setlength\itemsep{.5em}
 \item[\circled{1}]	In region $R_1=\{(w,s): 0<s<1/\cM, w<\nvM s\}$, we compare with the system (\ref{S}$^{++}$). The curve $C_1$ is a trajectory of (\ref{S}$^{++}$) that passes through (0,0). It exits the region at $(w_1,s_1)$, where $s_1=1/\cM$.
 \item[\circled{2}] In region $R_2=\{(w,s):s>1/\cM, w<\nvm s\}$, we compare with the system (\ref{S}$^{+-}$). The curve $C_2$ is a trajectory of (\ref{S}$^{+-}$) that passes through $(w_1,s_1)$. There are two scenarios:
 \begin{itemize}
  \item[(I).] If $\nvm\geq 2\sqrt{k\cM}$, $C_2$ won't exit the region $R_2$. The subcritical region $\ctsub$ consists the area in $\R\times\R_+$ to the right of the curve $C_1\cup C_2$.
  \item[(II).] If $0<\nvm< 2\sqrt{k\cM}$, $C_2$ exits the region at $(w_2,s_2)$, where $w_2 = \nvm s_2$. We continue our construction, under the admissible condition
  \begin{equation}\label{AC1}\tag{AC1}
	 s_2>\frac{1}{\cm}. 	
  \end{equation}  
 \end{itemize} 
 \item[\circled{3}] In region $R_3=\{(w,s):s>1/\cm, w>\nvm s\}$, we compare with the system (\ref{S}$^{--}$). The curve $C_3$ is a trajectory of (\ref{S}$^{--}$) that passes through $(w_2,s_2)$. It exits the region at $(w_3,s_3)$, where $s_3=1/\cm$. In order to continue to the next step, we require the admissible condition 
 \begin{equation}\label{AC2}\tag{AC2}
	w_3>\frac{\nvM}{\cm}, 	
 \end{equation}
     so that $(w_3,s_3)$ lies at the boundary of the next region $R_4$.
 \item[\circled{4}] In region $R_4=\{(w,s): 0<s<1/\cm, w>\nvM s\}$, we compare with the system (\ref{S}$^{-+}$). The curve $C_4$ is a trajectory of (\ref{S}$^{-+}$) that passes through $(w_3,s_3)$. To close the construction, we require another admissible condition such that $C_4$ reaches $(w_*, 0)$ with 
  \begin{equation}\label{AC3}\tag{AC3}
	w_*\geq0. 	
 \end{equation}
 \item[] The subcritical region $\ctsub$ is the open set enclosed by $C_1\cup C_2\cup C_3\cup C_4$ and $\{s=0\}$.
\end{itemize}

\begin{figure}[h!]
    \includegraphics{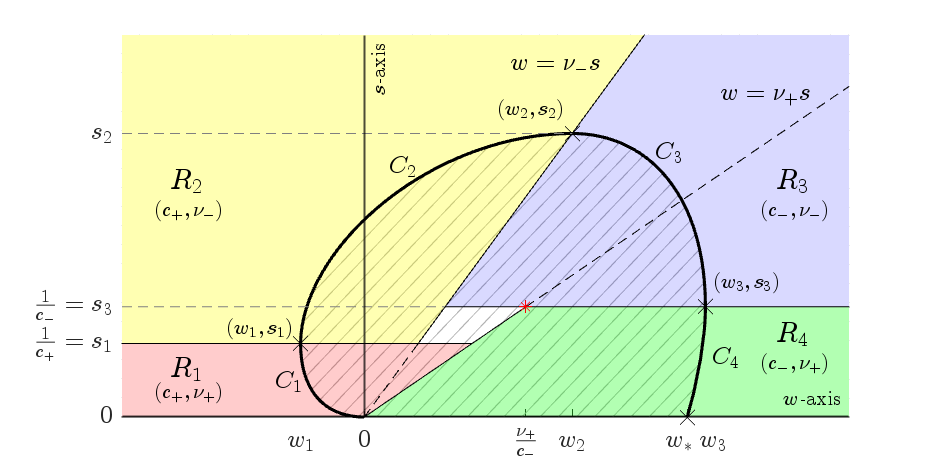}
\caption{Illustration of the subcritical region $\ctsub$ (shaded area)}\label{fig:sub}
\end{figure}

Figure \ref{fig:sub} illustrates the construction. The curves $\{C_i\}_{i=1}^4$ are trajectories of the localized systems in the corresponding regions $\{R_i\}_{i=1}^4$. More explicit expressions for these curves and the subcritical region $\ctsub$ will be provided in Section \ref{sec:Lyapunov}.

We comment on the two scenarios. Scenario I, referred to as \emph{strong alignment}, results in a large unbounded subcritical region $\ctsub$, determined by $C_1$ and $C_2$. 
In contrast, Scenario II, corresponding to \emph{median} and \emph{weak alignment}, leads to a subcritical region $\ctsub$ that is a bounded set. Moreover, the successful construction of $\ctsub$ depends on the three admissible conditions \eqref{AC1}-\eqref{AC3}. The explicit expressions for these conditions will be provided in \eqref{AC1e}-\eqref{AC3e}. For more discussion on the admissible conditions, see Remark \ref{rem:AC}.

\subsection{Supercritical region}
The construction of the supercritical region $\ctsup$ follows a similar approach to that of the subcritical region $\ctsub$. The key difference lies in applying comparison principles with different localized systems in the appropriate regions, as outlined below. For clarity, all quantities related to the supercritical region are denoted with a tilde.

\begin{itemize}  
\setlength\itemsep{.5em}
 \item[\circled{1}]	In region $\Rt_1=\{(w,s): 0<s<1/\cm, w<\nvm s\}$, we compare with the system (\ref{S}$^{--}$). The curve $\Ct_1$ is a trajectory of (\ref{S}$^{--}$) that passes through (0,0). It exits the region at $(\wt_1,\st_1)$, where $\st_1=1/\cm$.
 \item[\circled{2}] In region $\Rt_2=\{(w,s):s>1/\cm, w<\nvM s\}$, we compare with the system (\ref{S}$^{-+}$). The curve $\Ct_2$ is a trajectory of (\ref{S}$^{-+}$) that passes through $(\wt_1,\st_1)$. There are two scenarios:
 \begin{itemize}
  \item[(III).] If $\nvM\geq 2\sqrt{k\cm}$, $\Ct_2$ won't exit the region $\Rt_2$. The supercritical region $\ctsup$ consists the area in $\R\times\R_+$ to the left of the curve $\Ct_1\cup \Ct_2$.
  \item[(IV).] If $0<\nvM< 2\sqrt{k\cm}$, $\Ct_2$ exits the region at $(\wt_2,\st_2)$, where $\wt_2 = \nvM \st_2$. We continue our construction.
 \end{itemize} 
 \item[\circled{3}] In region $\Rt_3=\{(w,s):s>1/\cM, w>\nvM s\}$, we compare with the system (\ref{S}$^{++}$). The curve $C_3$ is a trajectory of (\ref{S}$^{++}$) that passes through $(\wt_2,\st_2)$. It exits the region at $(\wt_3,\st_3)$, where $\st_3=1/\cM$.
 \item[\circled{4}] In region $\Rt_4=\{(w,s): 0<s<1/\cM, w>\nvm s\}$, we compare with the system (\ref{S}$^{+-}$). The curve $\Ct_4$ is a trajectory of (\ref{S}$^{+-}$) that passes through $(\wt_3,\st_3)$. $C_4$ will reach the $w$-axis at $(\wt_*, 0)$ with $\wt_*\geq0$.  
 \item[] The supercritical region $\ctsup$ is the complement of the open set enclosed by $\Ct_1\cup \Ct_2\cup \Ct_3\cup \Ct_4$ and $\{s=0\}$.
\end{itemize}

Scenario IV is referred to as \emph{weak alignment}. As in the construction of the subcritical region $\Sigma_\flat$, the construction of the supercritical region $\Sigma_\sharp$ involves admissible conditions analogous to \eqref{AC1}-\eqref{AC3}, which take the form:
\begin{equation}\label{eq:ACsup}
 \st_2>\frac1\cM ,\quad \wt_3 > \frac{\nvm}{\cM}, \quad \text{and} \quad \wt_* \geq 0.
\end{equation}
We will show in Propositions~\ref{prop:superAC1}, ~\ref{prop:superAC2} and \ref{prop:superAC3} that the conditions in \eqref{eq:ACsup} are satisfied automatically. Therefore, no admissible conditions are required to construct $\Sigma_\sharp$. See Figure~\ref{fig:sup} for an illustration of the supercritical region $\ctsup$.

\begin{figure}[h!]
    \includegraphics{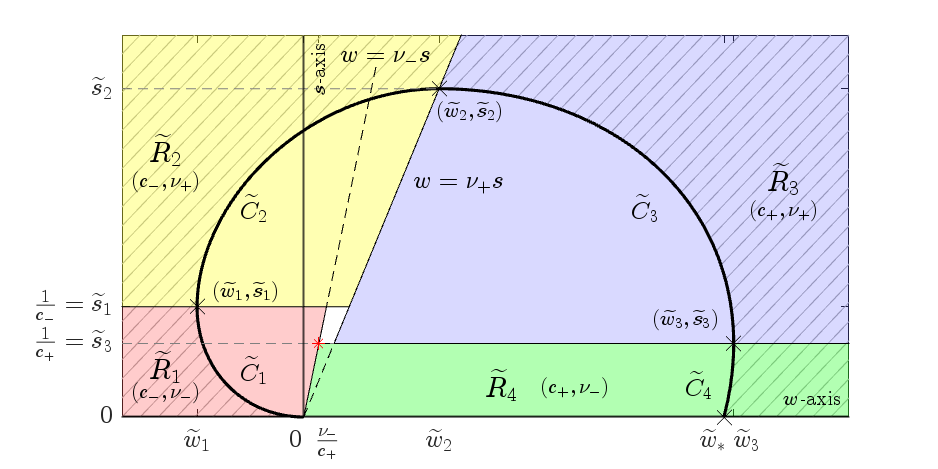}
\caption{Illustration of the supercritical region $\ctsup$ (shaded area)}\label{fig:sup}
\end{figure}

Any scenario that does not fall under strong or weak alignment is termed \emph{median alignment}, corresponding to the case where both II and III hold.

\subsection{Lyapunov functions}\label{sec:Lyapunov}
To express the critical thresholds $\ctsub$ and $\ctsup$ more explicitly, we adopt the approach introduced in \cite{bhatnagar2020critical2} for the damped Euler-Poisson equations, namely the EPA system with constant $c$ and $\nu$. This idea was later extended in \cite{choi2024critical} to accommodate variable background states by introducing a set of Lyapunov functions. These functions allow for an explicit characterization of the threshold regions and serve as a foundation for establishing the comparison principles.

Let us introduce the following type of \emph{Lyapunov functions}
\begin{equation}\label{eq:L}
	\L(w,s) = w - \nu s+\sqrt{2P(s)},	
\end{equation}
where the function $P$ satisfies the ODE:
\begin{equation}\label{eq:P}
     \frac{\dd P}{\dd s} = \nu\sqrt{2P(s)}+k(1-cs),  
\end{equation}
with an appropriately chosen initial condition to be specified later.

The following theorem establishes the connection between the Lyapunov functions and the trajectories of the localized dynamics: the trajectories coincide with the zero level sets of the Lyapunov functions.
\begin{theorem}\label{thm:levelset}
 Let $c$ and $\nu$ be constants. Let $(w(t),s(t))$ be a solution to \eqref{S} in a time interval $I$. Suppose $\L$ and $P$ satisfy \eqref{eq:L} and \eqref{eq:P} respectively, and 
 \[\textnormal{Range}\{s(t):t\in I\}\subseteq\textnormal{Dom}(P).\]
 If $\L(w(t_0),s(t_0))=0$ for some $t_0\in I$, then we must have  
 \[\L(w(t),s(t)) = 0,\quad\forall~t\in I.\]
\end{theorem}
\begin{proof}
  Apply \eqref{S}, \eqref{eq:L}, \eqref{eq:P} and compute:
  \begin{align*}
    \frac{\dd}{\dd t}\L(w(t),s(t)) & = w'(t)-\nu s'(t) + \frac{1}{\sqrt{2P(s)}}\cdot \frac{\dd P}{\dd s}\cdot s'(t)\\
	& = k\big(1-cs(t)\big) - \left(\nu-\frac{\nu\sqrt{2P(s(t))}+k\big(1-cs(t)\big)}{\sqrt{2P(s(t))}}\right)\big(w(t)-\nu s(t)\big)\\
	& = \frac{k\big(1-cs(t)\big)}{\sqrt{2P(s(t))}}\,\L(w(t),s(t)).
  \end{align*}
  This implies
  \[\L(w(t),s(t)) = \L(w(t_0),s(t_0))\int_{t_0}^t\frac{k\big(1-cs(\tau)\big)}{\sqrt{2P(s(\tau))}}\,\dd\tau=0.\]
\end{proof}

Another family of Lyapunov functions satisfying the same property in Theorem \ref{thm:levelset} is 
\begin{equation}\label{eq:L2}
	\L(w,s) = w - \nu s - \sqrt{2N(s)},	
\end{equation}
where the function $N$ satisfies the ODE:
\begin{equation}\label{eq:N}
     \frac{\dd N}{\dd s} = - \nu\sqrt{2N(s)}+k(1-cs). 
\end{equation}

We will employ these Lyapunov functions for the localized systems (\ref{S}$^{\pm\pm}$), using the notations $\L^{\pm\pm}$, $P^{\pm\pm}$ and $N^{\pm\pm}$ to denote the Lyapunov function associated with the parameters $(c_\pm, \nu_\pm)$. For example,
\[\L^{+-}(w,s) = w - \nvm s+\sqrt{2P^{+-}(s)}, \quad\text{and}\quad
\frac{\dd P^{+-}}{\dd s} = \nvm\sqrt{2P^{+-}(s)}+k(1-\cM s).\]

Now, we are ready to use the Lyapunov functions to characterize the threshold regions. To proceed, we construct the Lyapunov functions region by region, selecting appropriate initial conditions for the function $P$ or $N$ to ensure that the trajectory forms a zero level set of $\L$. We denote the corresponding Lyapunov functions and functions $P, N$ in the four regions using subscripts $1$ through $4$.

For the subcritical region:
\begin{itemize}  
\setlength\itemsep{.5em}
 \item[\circled{1}]	$C_1$ can be represented as the level set $\L_1^{++}(w,s)=0$, with $\L$ defined in \eqref{eq:L}, namely
\[
 C_1=\left\{(w,s):w=\nvM s-\sqrt{2P_1^{++}(s)},~~s\in[0,s_1]\right\}.
\]
The function $P_1^{++}$ satisfies
\begin{equation}\label{aux_P_1}
   \frac{\dd P^{++}_1}{\dd s} = \nvM\sqrt{2P_1^{++}(s)}+k(1-\cM s),\quad P_1^{++}(0) = 0,
\end{equation}
where the initial condition is chosen so that $\L_1^{++}(0,0)=\sqrt{2P^{++}_1(0)}=0$. The exit point is
\[
(w_1,s_1)=\left(\tfrac{\nvM}{\cM}-\sqrt{2P_1^{++}(\tfrac{1}{\cM})},~~ \tfrac{1}{\cM}\right).
\]
 \item[\circled{2}] $C_2$ can be represented as the level set $\L_2^{+-}(w,s)=0$, with $\L$ defined in \eqref{eq:L}, namely
\[
 C_2=\left\{(w,s):w=\nvm s-\sqrt{2P_2^{+-}(s)},~~s\in[s_1,s_2]\right\}.
\]
The function $P_2^{+-}$ satisfies
\begin{equation}\label{aux_P_2}
   \frac{\dd P_2^{+-}}{\dd s} = \nvm\sqrt{2P_2^{+-}(s)}+k(1-\cM s),\quad P_2^{+-}(s_1) = \frac12(w_1-\nvm s_1)^2, 
\end{equation}
where the initial condition is chosen so that $\L_2^{+-}(w_1,s_1)=0$.

We will show in Lemmas \ref{lem:P2strong} and \ref{lem:P2weak} that the lifespan of $P_2^{+-}$ in \eqref{aux_P_2} is
\begin{equation}\label{dom:P2pm}
	\textnormal{Dom}(P_2^{+-}) = \begin{cases}
	[s_1,\infty) & \nvm \ge 2\sqrt{k\cM},\\
	[s_1,s_2] & 0 < \nvm < 2\sqrt{k\cM}.
	\end{cases}
\end{equation}
The two cases correspond to Scenarios I and II, respectively. In the latter case, the explicit expression of the exit point $s_2$ is given in \eqref{eq:s2}, and correspondingly 
\[w_2 = \nvm s_2 - \sqrt{2P_2^{+-}(s_2)}=\nvm s_2.\]

\item[\circled{3}] $C_3$ can be represented as the level set $\L_3^{--}(w,s)=0$, with $\L$ defined in \eqref{eq:L2}, namely
\[
 C_3=\left\{(w,s):w=\nvm s+\sqrt{2N_3^{--}(s)},~~s\in[s_3,s_2]\right\}.
\]
The function $N_3^{--}$ satisfies
\begin{equation}\label{aux_N_3}
   \frac{\dd N_3^{--}}{\dd s} = -\nvm\sqrt{2N_3^{--}(s)}+k(1-\cm s),\quad N_3^{--}(s_2) = 0,
\end{equation}
where the initial condition is chosen so that $\L_3^{--}(w_2,s_2)=-\sqrt{2N_3^{--}(s_2)}=0$. $N_3^{--}$ propagates backward in $s$. The exit point 
\[
(w_3,s_3)=\left(\tfrac{\nvm}{\cm}+\sqrt{2N_3^{--}(\tfrac{1}{\cm})},~~ \tfrac{1}{\cm}\right).
\]
 
 \item[\circled{4}] $C_4$ can be represented as the level set $\L_4^{-+}(w,s)=0$, with $\L$ defined in \eqref{eq:L2}, namely
\[
 C_4=\left\{(w,s):w=\nvM s+\sqrt{2N_4^{-+}(s)},~~s\in[0,s_3]\right\}.
\]
The function $N_4^{-+}$ satisfies
\begin{equation}\label{aux_N_4}
   \frac{\dd N_4^{-+}}{\dd s} = -\nvM\sqrt{2N_4^{-+}(s)}+k(1-\cm s),\quad N_4^{-+}(s_3) = \frac12(w_3-\nvM s_3)^2,
\end{equation}
where the initial condition is chosen so that $\L_4^{--}(w_3,s_3)=0$. $N_4^{-+}$ propagates backward in $s$. Its lifespan is
\begin{equation}\label{dom:P4mp}
\textnormal{Dom}(N_4^{-+}) = \begin{cases}
(-\infty,s_3] \ \ \ \nvM \ge 2\sqrt{k\cm},\\
[s_4,s_3] \ \ \  0 < \nvM < 2\sqrt{k\cm}.
\end{cases}
\end{equation}
See Lemmas \ref{lem:N4strong} and \ref{lem:N4weak} for the proof.
The admissible condition \eqref{AC3} is equivalent to $s_4\leq0$ (see Lemma \ref{lem:AC3}).
\end{itemize}

We assemble the four regions together and express the subcritical region $\ctsub$ using the functions $P$ and $N$.
\begin{align*}
 \textnormal{Senario I:}\quad & \ctsub = \big\{(w,s):w > W_\ell(s),\,\, s>0\big\},\\
 \textnormal{Senario II:}\quad & \ctsub = \big\{(w,s): W_\ell(s) < w < W_r(s),\,\, 0<s<s_2\big\},	
\end{align*}
where
\begin{align*}
 W_\ell(s):=\begin{cases}
\nvM s-\sqrt{2P_1^{++}(s)}, & 0 < s \le \frac{1}{\cM},\\
\nvm s-\sqrt{2P_2^{+-}(s)}, & s > \frac{1}{\cM},
\end{cases}\\
 W_r(s):=\begin{cases}
\nvM s+\sqrt{2N_4^{-+}(s)}, & 0 < s \le \frac{1}{\cm},\\
\nvm s+\sqrt{2N_3^{--}(s)}, & s > \frac{1}{\cm}.
\end{cases}
\end{align*}

The subcritical region $\CTsub$ in the $(G,\rho)$ space can be explicitly derived by employing the relation \eqref{eq:ws}.

\begin{theorem}[{\bf Global regularity}]\label{thm:sub}
Consider the EPA system \eqref{eq:main} under the assumptions in Theorem \ref{thm:LWP}. Suppose the initial data $(\rho_0, G_0)$ lies within the following subcritical regions:
\begin{itemize}
    \item \textnormal{Strong alignment $(\nvm \geq 2\sqrt{k\cM})$:} for any $x\in\T$, $(G_0(x),\rho_0(x))$ satisfies
	\[
	G_0(x) > \begin{cases}
 	\nvm-\rho_0(x)\sqrt{2P_2^{+-}\left(\tfrac{1}{\rho_0(x)}\right)},& \text{if}\,\,\, 0<\rho_0(x)<\cM,\\
 	\nvM-\rho_0(x)\sqrt{2P_1^{++}\left(\tfrac{1}{\rho_0(x)}\right)},& \text{if}\,\,\, \rho_0(x)\ge\cM.
 \end{cases}
    \]
    \item \textnormal{Weak or median alignment $(\nvm < 2\sqrt{k\cM})$:} under admissible conditions \eqref{AC1}-\eqref{AC3}, for any $x\in\T$, we require 
    \[\rho_0(x)>\tfrac{1}{s_2},\]
    and $(G_0(x),\rho_0(x))$ satisfies
    \begin{align*}
        &G_0(x) > \begin{cases}
 	\nvm-\rho_0(x)\sqrt{2P_2^{+-}\left(\tfrac{1}{\rho_0(x)}\right)},& \text{if}\,\,\, \tfrac{1}{s_2}<\rho_0(x)<\cM,\\
 	\nvM-\rho_0(x)\sqrt{2P_1^{++}\left(\tfrac{1}{\rho_0(x)}\right)},& \text{if}\,\,\, \rho_0(x)\geq\cM,
 \end{cases} \quad \text{and} \\
        &G_0(x) < \begin{cases}
 	\nvm+\rho_0(x)\sqrt{2N_3^{--}\left(\tfrac{1}{\rho_0(x)}\right)},& \text{if}\,\,\, \tfrac{1}{s_2}<\rho_0(x)<\cm,\\
 	\nvM+\rho_0(x)\sqrt{2N_4^{-+}\left(\tfrac{1}{\rho_0(x)}\right)},& \text{if}\,\,\, \rho_0(x)\geq\cm.
 \end{cases}
   \end{align*}
\end{itemize}
Then the solution $(\rho, G)$ remain bounded in all time. Consequently, the solution $(\rho,G)$ is globally regular, in the sense of \eqref{eq:regrhoG}.
\end{theorem}

\begin{remark}[Admissible conditions]\label{rem:AC}
  In weak or median alignment cases, admissible conditions \eqref{AC1}-\eqref{AC3} are needed in order to construct the subcritical region. We state some properties of these conditions. See Remarks \ref{rem:AC1}, \ref{rem:AC2} and \ref{rem:AC3} for detailed discussions.
  \begin{enumerate}
   \item Among the three admissible conditions, \eqref{AC3} is the strongest: it implies \eqref{AC2}, which in turn implies \eqref{AC1}.
   \item For median alignment, the admissible condition \eqref{AC2} is sufficient. For weak alignment, the stronger condition \eqref{AC3} is required. 
   \item When $c$ is constant, \eqref{AC1} is automatically satisfied. In this case, our admissible conditions \eqref{AC2} and \eqref{AC3} coincide with those in \cite[Theorem 2.2]{bhatnagar2023critical} for the median and weak alignment scenarios, respectively.
   \item When $\nu$ is constant, \eqref{AC2} holds automatically provided that \eqref{AC1} is satisfied. In this case, our admissible conditions \eqref{AC1} and \eqref{AC3} reduce to the closing conditions in \cite[Theorem 1.9]{choi2024critical} for their cases \#2.1 and \#2.2, respectively.
  \end{enumerate}
\end{remark}

For the supercritical region, a similar argument would lead to the description of the supercritical region $\ctsup$ as follows.
\begin{align*}
 \textnormal{Senario III:}\quad & \ctsup = \big\{(w,s):w \leq \widetilde{W}_\ell(s),\,\, s>0\big\},\\
 \textnormal{Senario IV:}\quad & \ctsup = (\R\times\R_+)\backslash\big\{(w,s): \widetilde{W}_\ell(s) < w < \widetilde{W}_r(s),\,\, 0<s<\st_2\big\},	
\end{align*}
where
\begin{align*}
 \widetilde{W}_\ell(s):= & \begin{cases}
\nvm s-\sqrt{2\Pt_1^{--}(s)}, & 0 < s \le \frac{1}{\cm},\\
\nvM s-\sqrt{2\Pt_2^{-+}(s)}, & s > \frac{1}{\cm},
\end{cases}\\
 \widetilde{W}_r(s):= & \begin{cases}
\nvm s+\sqrt{2\Nt_4^{+-}(s)}, & 0 < s \le \frac{1}{\cM},\\
\nvM s+\sqrt{2\Nt_3^{++}(s)}, & s > \frac{1}{\cM},
\end{cases}
\end{align*}
and the functions $\Pt$ and $\Nt$ satisfy
\begin{align}
   & \frac{\dd \Pt^{--}_1}{\dd s} = \nvm\sqrt{2\Pt_1^{--}(s)}+k(1-\cm s),\quad \Pt_1^{--}(0) = 0,\\
   & \frac{\dd \Pt_2^{-+}}{\dd s} = \nvM\sqrt{2P_2^{-+}(s)}+k(1-\cm s),\quad \Pt_2^{+-}(\st_1) = \frac12(\wt_1-\nvM \st_1)^2, \\
   & \frac{\dd \Nt_3^{++}}{\dd s} = -\nvM\sqrt{2\Nt_3^{++}(s)}+k(1-\cM s),\quad \Nt_3^{++}(\st_2) = 0,\\
   & \frac{\dd \Nt_4^{+-}}{\dd s} = -\nvm\sqrt{2\Nt_4^{+-}(s)}+k(1-\cM s),\quad \Nt_4^{+-}(\st_3) = \frac12(\wt_3-\nvm \st_3)^2.\label{eq:Nt4}
\end{align}

\begin{theorem}[{\bf Finite-time blowup}]\label{thm:sup}
Consider the EPA system \eqref{eq:main} under the assumptions in Theorem \ref{thm:LWP}. Suppose the initial data $(\rho_0, G_0)$ lies within the following supercritical regions:
\begin{itemize}
    \item \textnormal{Strong and median alignment $(\nvm \geq 2\sqrt{k\cM})$:} there exists $x_*\in\T$, such that $(G_0(x_*),\rho_0(x_*))$ satisfies
	\[
	G_0(x_*) \leq\begin{cases}
 	\nvM-\rho_0(x_*)\sqrt{2\Pt_2^{-+}\left(\tfrac{1}{\rho_0(x_*)}\right)},& \text{if}\,\,\, 0<\rho_0(x_*)<\cm,\\
 	\nvm-\rho_0(x_*)\sqrt{2\Pt_1^{--}\left(\tfrac{1}{\rho_0(x_*)}\right)},& \text{if}\,\,\, \rho_0(x_*)\ge\cm.
 \end{cases}
    \]
    \item \textnormal{Weak alignment $(\nvm < 2\sqrt{k\cM})$:} there exists $x_*\in\T$, such that
    \[\rho_0(x_*)\leq \tfrac{1}{\st_2},\]
	or $(G_0(x_*),\rho_0(x_*))$ satisfies
    \begin{align*}
        &G_0(x_*) \leq \begin{cases}
 	\nvM-\rho_0(x_*)\sqrt{2\Pt_2^{-+}\left(\tfrac{1}{\rho_0(x_*)}\right)},& \text{if}\,\,\, \tfrac{1}{\st_2}<\rho_0(x_*)<\cm,\\
 	\nvm-\rho_0(x_*)\sqrt{2\Pt_1^{--}\left(\tfrac{1}{\rho_0(x_*)}\right)},& \text{if}\,\,\, \rho_0(x_*)\geq\cm,
 \end{cases} \quad \text{or} \\
        &G_0(x_*) \geq \begin{cases}
 	\nvM+\rho_0(x_*)\sqrt{2\Nt_3^{++}\left(\tfrac{1}{\rho_0(x_*)}\right)},& \text{if}\,\,\, \tfrac{1}{\st_2}<\rho_0(x_*)<\cM,\\
 	\nvm+\rho_0(x_*)\sqrt{2\Nt_4^{+-}\left(\tfrac{1}{\rho_0(x_*)}\right)},& \text{if}\,\,\, \rho_0(x_*)\geq\cM.
 \end{cases}
   \end{align*}
\end{itemize}
Then the solution along the characteristic path $(\rho(X(t; x_*),t), G(X(t; x_*),t))$ becomes unbounded in finite time. Consequently, the solution generates a singular shock at a finite time $T_*$, in the sense of \eqref{eq:singularshock}.
\end{theorem}

Figure \ref{fig:ct} illustrates the subcritical region $\Sigma_\flat$ and the supercritical region $\Sigma_\sharp$ under three different alignment strengths: weak, median, and strong. These regions are mapped to their counterparts in Figure \ref{fig:CT} via the transformation given in \eqref{eq:ws}.
\begin{figure}[h!]
\centering
\includegraphics{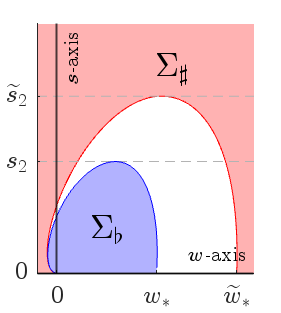}	
\includegraphics{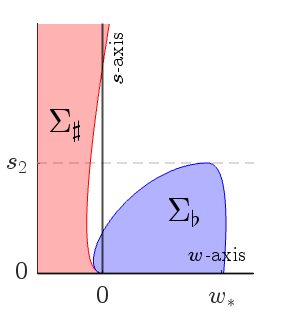}	
\includegraphics{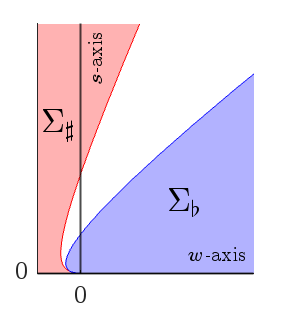}
\caption{Illustration of the critical thresholds $\ctsub$ and $\ctsup$. Left: weak alignment; middle: median alignment; right: strong alignment.}\label{fig:ct}
\end{figure}

\section{Proof of main results}\label{sec:proof}
This section is devoted to proving our main critical threshold results, Theorems \ref{thm:sub} and \ref{thm:sup}, using Lyapunov functions effectively.

\subsection{Comparison principles}\label{sec:comparison}
The main analytical tool for proving the critical thresholds is the \emph{comparison principle}. Here, we present a version of the comparison principle based on the Lyapunov functions $\L$ introduced in \eqref{eq:L}. Compared to the classical approach of comparing trajectories $(w(t),s(t))$ with those of the localized system, as used in \cite{bhatnagar2020critical, bhatnagar2023critical}, the Lyapunov function method provides a more elegant and direct comparison. This approach has been successfully applied in \cite{choi2024critical} to the damped Euler-Poisson equation with variable background. Here, we extend this idea to the general EPA system \eqref{eq:main}.

First, we establish a set of weak comparison principles for analyzing the supercritical threshold region $\ctsup$.

\begin{proposition}[Weak comparison principles]\label{prop:wCP}
Let $(w(t),s(t))$ be a solution to the system \eqref{S}. Denote by $C$ the trajectory in the $(w,s)$-plane over the time interval $(t_0,t_1)$, namely
\begin{equation}\label{eq:C}
C:=\big\{(w(t),s(t)):t\in (t_0,t_1)\big\}.	
\end{equation}
Then, the following results hold:
\begin{itemize}
\item[\circled{1}] Suppose $C\subseteq \Rt_1$. Then
	\begin{equation}\label{eq:sup1}
    \Lt_1^{--}(w(t_0),s(t_0)) \leq 0 \quad\textnormal{implies}\quad \Lt_1^{--}(w(t),s(t)) \leq 0, \quad\forall~t\in[t_0,t_1].
	\end{equation}
\item[\circled{2}] Suppose $C\subseteq \Rt_2$ in \textnormal{Senario III}, or $C\subseteq \Rt_2\,\cap\,\{(w,s):s\leq \st_2\}$ in \textnormal{Senario IV}. Then
	\begin{equation}\label{eq:sup2}
    \Lt_2^{-+}(w(t_0),s(t_0)) \leq 0 \quad\textnormal{implies}\quad \Lt_2^{-+}(w(t),s(t)) \leq 0, \quad\forall~t\in[t_0,t_1].
	\end{equation}
\item[\circled{3}] In \textnormal{Senario IV}, suppose $C\subseteq \Rt_3\,\cap\,\{(w,s):s\leq \st_2\}$. Then
	\begin{equation}\label{eq:sup3}
    \Lt_3^{++}(w(t_0),s(t_0)) \geq 0 \quad\textnormal{implies}\quad \Lt_3^{++}(w(t),s(t)) \geq 0, \quad\forall~t\in[t_0,t_1].
	\end{equation}
\item[\circled{4}] In \textnormal{Senario IV}, suppose $C\subseteq \Rt_4$. Then
	\begin{equation}\label{eq:sup4}
    \Lt_4^{+-}(w(t_0),s(t_0)) \geq 0 \quad\textnormal{implies}\quad \Lt_4^{+-}(w(t),s(t)) \geq 0, \quad\forall~t\in[t_0,t_1].
	\end{equation}
\end{itemize}
\end{proposition}
\begin{proof}
We start with the proof of \eqref{eq:sub1}. Compute
  \begin{align*}
    & \frac{\dd}{\dd t}\Lt_1^{--}(w(t),s(t)) = w'(t)-\nvm s'(t) + \frac{1}{\sqrt{2\Pt_1^{--}(s)}}\cdot \frac{\dd \Pt_1^{--}}{\dd s}\cdot s'(t)\\
	& = k\big(1-c(t)s(t)\big) - \left(\nvm-\frac{\nvm\sqrt{2\Pt_1^{--}(s(t))}+k\big(1-\cm s(t)\big)}{\sqrt{2\Pt_1^{--}(s(t))}}\right)\big(w(t)-\nu s(t)\big)\\
	& = \frac{k\big(1-\cm s(t)\big)}{\sqrt{2\Pt_1^{--}(s(t))}}\,\Lt_1^{--}(w(t),s(t)) - k\big(c(t)-\cm\big)s(t)-k\big(\nu(t)-\nvm\big)\big(1-\cm s(t)\big)s(t)\\
	& \leq \frac{k\big(1-\cm s(t)\big)}{\sqrt{2\Pt_1^{--}(s(t))}}\,\Lt_1^{--}(w(t),s(t)).
  \end{align*}
In the last inequality, we have used the fact $0<s(t)\leq \frac{1}{\cm}$ (since $C\in\Rt_1$) along with the bounds \eqref{eq:nuc} on $c$ and $\nu$.
Applying Gr\"{o}nwall's inequality yields:
\[\Lt_1^{--}(w(t),s(t))\leq\Lt_1^{--}(w(t_0),s(t_0))\exp\left(\int_{t_0}^t \frac{k\big(1-\cm s(\tau)\big)}{\sqrt{2\Pt_1^{--}(s(\tau))}}\,\dd\tau\right)\leq0,\]
provided that $\Lt_1^{--}(w(t_0),s(t_0))\leq0$. This completes the proof of the comparison principle \eqref{eq:sup1}.

The other comparison principles follow from the same argument. The additional assumptions of $s\leq\st_2$ for the second and third regions ensure that the Lyapunov function is well-defined. This assumption is only necessary in Scenario IV (the weak alignment case). We omit the proofs.
\end{proof}

Next, we introduce a set of strong comparison principles for examining the subcritical threshold region $\ctsub$.
\begin{proposition}[Strong comparison principles]\label{prop:sCP}
Let $(w(t),s(t))$ be a solution to the system \eqref{S}. The trajectory $C$  over the time interval $(t_0,t_1)$ is defined as in \eqref{eq:C}.
Then, the following results hold:
\begin{itemize}
\item[\circled{1}] Suppose $C\subseteq R_1$. Then
	\begin{equation}\label{eq:sub1}
    \L_1^{++}(w(t_0),s(t_0)) > 0 \quad\textnormal{implies}\quad \L_1^{++}(w(t),s(t)) > 0, \quad\forall~t\in[t_0,t_1].
	\end{equation}
\item[\circled{2}] Suppose $C\subseteq R_2$. Then
	\begin{equation}\label{eq:sub2}
    \L_2^{+-}(w(t_0),s(t_0)) > 0 \quad\textnormal{implies}\quad \L_2^{+-}(w(t),s(t)) > 0, \quad\forall~t\in[t_0,t_1].
	\end{equation}
\item[\circled{3}] In \textnormal{Scenario II}, suppose $C\subseteq R_3$. Then
	\begin{equation}\label{eq:sub3}
    \L_3^{--}(w(t_0),s(t_0)) < 0 \quad\textnormal{implies}\quad \L_3^{--}(w(t),s(t)) < 0, \quad\forall~t\in[t_0,t_1].
	\end{equation}
\item[\circled{4}] In \textnormal{Scenario II}, suppose $C\subseteq R_4$. Then
	\begin{equation}\label{eq:sub4}
    \L_4^{-+}(w(t_0),s(t_0)) < 0 \quad\textnormal{implies}\quad \L_4^{-+}(w(t),s(t)) < 0, \quad\forall~t\in[t_0,t_1].
	\end{equation}
\end{itemize}
\end{proposition}
\begin{proof}
We start with the proof of \eqref{eq:sub1}. Following the same argument as in Proposition \ref{prop:wCP}, we obtain the bound
\[\L_1^{++}(w(t),s(t))\geq\L_1^{++}(w(t_0),s(t_0))\exp\left(\int_{t_0}^t \frac{k\big(1-\cM s(\tau)\big)}{\sqrt{2P_1^{++}(s(\tau))}}\,\dd\tau\right).\]
Since $C\subseteq R_1$, we have $s(t)\in(0,\frac{1}{\cM})$ for any $t\in(t_0,t_1)$. Therefore, the integral on the right-hand side of the inequality is positive. This leads to the strict inequality
\[\L_1^{++}(w(t),s(t))\geq\L_1^{++}(w(t_0),s(t_0))>0.\]
The proof of \eqref{eq:sub3} follows in a similar manner.

Next, we turn to the proof of \eqref{eq:sub2}. Following the same argument as in Proposition \ref{prop:wCP}, we obtain the bound
\[
\L_2^{+-}(w(t),s(t)) \geq\L_2^{+-}(w(t_0),s(t_0))\exp\left(-\int_{t_0}^t \frac{k\big(\cM s(\tau)-1\big)}{\sqrt{2P_2^{+-}(s(\tau))}}\,\dd\tau\right).
\]
We claim that there is a uniform positive lower bound 
\begin{equation}\label{eq:P2low}
 P_2^{+-}(s(\tau))\geq p>0,\quad \forall~\tau\in[t_0,t_1].	
\end{equation}
To show \eqref{eq:P2low}, we first observe that $s'(t)\leq0$ in the region $R_2$. Hence, $s(\tau)\in[s(t_1),s(t_0)]$. In Scenario I, the uniform bound \eqref{eq:P2low} follows from Lemma \ref{lem:P2strong}. In Scenario II, from the definition \eqref{eq:L} for $\L_2^{+-}$, we have 
\[\sqrt{2P_2^{+-}(s(t_0))}=\L_2^{+-}(w(t_0),s(t_0))-\big(w(t_0)-\nu_-s(t_0)\big)
\geq\L_2^{+-}(w(t_0),s(t_0))>0,\]
where we have used the fact that $w(t_0)\leq\nu_-s(t_0)$ as $C\subseteq R_2$. Therefore, we have the strict inequality $s(t_0)<s_2$. Then, the bound \eqref{eq:P2low} follows from Lemma \ref{lem:P2weak}.

Applying the bound \eqref{eq:P2low}, we end up with the strict inequality
\begin{align*}
\L_2^{+-}(w(t),s(t)) & \geq \L_2^{+-}(w(t_0),s(t_0))\exp\left(-(t_1-t_0)\frac{k\cM s(t_0)}{\sqrt{2p}}\right)>0,
\end{align*}
finishing the proof. 
The proof of \eqref{eq:sub4} follows in a similar manner. 
\end{proof}

\subsection{Global regularity}
In this section, we apply the strong comparison principles from Proposition \ref{prop:sCP} to establish global regularity for subcritical initial data, as stated in Theorem \ref{thm:sub}. The key idea is to demonstrate that the subcritical region $\ctsub$ remains an \emph{invariant region} under the dynamics \eqref{S}.

\begin{proposition}[Invariant subcritical region]\label{prop:sub}
  Consider the dynamics \eqref{S} with initial data $(w_0,s_0)\in \ctsub$. Then, for any $t\geq0$, the solution
	\[(w(t),s(t))\in \ctsub.\]
\end{proposition}
\begin{proof}
Since the solution $(w(t),s(t))$ of the dynamics \eqref{S} is continuous, the trajectory must intersect the boundary $\pa\ctsub$ if it were to leave the region $\ctsub$. We will establish a contradiction by showing that the solution cannot reach $\pa\ctsub$.

Recall that the boundary $\pa\ctsub$ consists of the following segments:
\begin{align*}
\text{Scenario I:} \quad & \pa\ctsub=C_1\cup C_2 \cup \big\{(w,s): w\geq0,s=0\big\},\\
\text{Scenario II:} \quad & \pa\ctsub=C_1\cup C_2\cup C_3\cup C_4 \cup \big\{(w,s): w\in[0,w_*],s=0\big\}.
\end{align*}

We apply strong comparison principles in Proposition \ref{prop:sCP} to show that the trajectory cannot reach $C_i$. Take $C_1$ as an example. Suppose the trajectory stays in $R_1$ in the time interval $(t_0,t_1)$. From the comparison principle \eqref{eq:sub1}, we obtain that 
	\[\L_1^{++}(w(t),s(t))>0, \quad\forall~t\in[t_0,t_1],\]
	provided that 
	\begin{equation}\label{eq:subinit}
		\L_1^{++}(w(t_0),s(t_0))>0.
	\end{equation} On the other hand, by the construction of $C_1$, we know
	\[\L_1^{++}(w,s)=0,\quad\forall~(w,s)\in C_1.\]
	Therefore, $(w(t),s(t))$ cannot reach $C_1$.
	
Next, we verify the assumption \eqref{eq:subinit}. If the trajectory is initiated in $R_1$, namely $t_0=0$, then $(w_0,s_0)\in\ctsub$ implies \eqref{eq:subinit}. If the trajectory enters $R_1$ at a later time $t_0>0$,  the only potential violation of \eqref{eq:subinit} occurs if $(w(t_0),s(t_0)) = (w_1,s_1)$. However, this contradicts the comparison principle in the region $R_2$. Therefore, \eqref{eq:subinit} must hold unless the trajectory reaches another segment of the boundary.

The same argument applies to $C_2$, $C_3$ and $C_4$.

Finally, we argue that 
\begin{equation}\label{eq:pos}
	s(t)>0,\quad \forall~t\geq0,	
\end{equation}
so the trajectory cannot reach the $w$-axis. Suppose the trajectory stays in $R_4$ in the time interval $(t_0,t_1)$. Observe that $s'(t)>0$ if $(w(t),s(t))\in R_4$. Hence,
	$s(t)>s(t_0)>0$ if $s(t_0)>0$. If the trajectory is initiated in $R_4$, namely $t_0=0$, then $(w_0,s_0)\in\ctsub$ implies \eqref{eq:pos}. If the trajectory enters $R_4$ at a later time $t_0>0$, the only potential violation of \eqref{eq:pos} occurs if $(w(t_0),s(t_0)) = (0,0)$. However, this contradicts the comparison principle in the region $R_1$. Therefore, we must have \eqref{eq:pos} and $w$-axis cannot be reached.

We conclude that the trajectory cannot reach $\pa\ctsub$. Hence, the solution has to stay inside the region $\ctsub$ in all time.
\end{proof}

Given that the solution  $(w(t),s(t))$  remains within the invariant region $\ctsub$ and, in particular, that \eqref{eq:pos} holds, we are now prepared to establish the global well-posedness result stated in Theorem \ref{thm:sub}.

\begin{proof}[Proof of Theorem \ref{thm:sub}]
	Given any $x\in\T$, consider the dynamics \eqref{S} along the characteristic path $X(t;x)$. Applying Proposition \ref{prop:sub}, we obtain that if $(w_0(x),s_0(x))\in\ctsub$, then $(w(X(t;x),t),s(X(t;x),t))\in\ctsub$, and in particular \eqref{eq:pos} holds, namely $s(X(t;x),t))>0$. Using the relation \eqref{eq:ws}, we deduce
	\[\rho(X(t;x),t)<\infty,\quad\text{and}\quad |G(X(t;x),t)|=\rho(X(t;x),t)|w(X(t;x),t)|<\infty,\]
	and from \eqref{eq:G} we obtain
	\[|\pa_xu(X(t;x),t)|\leq |G(X(t;x),t)|+\nvM<\infty.\]
	Thus, the characteristic paths $X(t;x)$ remain well-defined for all time, with $\{X(t;x):x\in\T\}=\T$. Consequently, $\|\rho(\cdot,t)\|_{L^\infty}$ and $\|G(\cdot,t)\|_{L^\infty}$ are bounded for all time. The regularity criterion \eqref{eq:BKM} holds. Global regularity then follows from Theorem \ref{thm:LWP}.
	
	The subcritical region $(w,s)\in\ctsub$ can be translated to $(G,\rho)\in\CTsub$ using the relation \eqref{eq:ws}. Thanks to the Lyapunov functions $\L^{\pm\pm}$, the region can be expressed explicitly in terms of $P^{\pm\pm}$ and $N^{\pm\pm}$, as stated in Theorem \ref{thm:sub}.
\end{proof}

\subsection{Finite-time singularity formation}
In this section, we prove Theorem \ref{thm:sup}, demonstrating that supercritical initial data leads to finite-time singularity formation.

First, we apply the weak comparison principles in Proposition \ref{prop:wCP} to establish the following result on the invariant region $\ctsup$: the trajectory cannot leave $\ctsup$ except hitting the $w$-axis.

\begin{proposition}[Invariant supercritical region]\label{prop:sup}
  Consider the dynamics \eqref{S} with initial data $(w_0,s_0)\in \ctsup$. Let $t>0$ such that $s(\tau)>0$ for any $\tau\in[0,t]$. Then,
	\[(w(t),s(t))\in\ctsup.\]
\end{proposition}
\begin{proof}
Since the solution $(w(t),s(t))$ of the dynamics \eqref{S} is continuous, the trajectory must intersect the boundary $\pa\ctsup$ if it were to leave the region $\ctsup$. Recall that the boundary $\pa\ctsup$ consists of the following segments:
\begin{align*}
\text{Scenario III:} \quad & \pa\ctsub=\Ct_1\cup \Ct_2 \cup \big\{(w,s): w\leq0,s=0\big\},\\
\text{Scenario IV:} \quad & \pa\ctsub=\Ct_1\cup \Ct_2\cup \Ct_3\cup \Ct_4 \cup \big\{(w,s): w\in(-\infty,0]\cup[\wt_*,\infty),s=0\big\}.
\end{align*}
We will establish a contradiction by showing that the solution cannot cross $\{\Ct_i\}_{i=1}^4$.

Take $\Ct_1$ as an example. Suppose the solution reaches $\Ct_1$ at time $t_0$. Then, we have $(w(t_0),s(t_0))\in\Ct_1$, and by the definition of $\Ct_1$,
\[\Lt_1^{--}(w(t_0),s(t_0))=0.\]
Applying the weak comparison principle \eqref{eq:sup1}, we deduce
\[\Lt_1^{--}(w(t),s(t))\leq0,\quad\forall~t\in[t_0,t_1],\]
as long as the trajectory $C$ defined in \eqref{eq:C} lies in $\Rt_1$. Therefore, the trajectory cannot cross $\Ct_1$, and will stay in $\ctsup\cap\Rt_1$ until leaving $\Rt_1$.

The same argument applies to $\Ct_2, \Ct_3$ and $\Ct_4$. We now comment on the additional assumption $s\leq \st_2$ in the comparison principles \eqref{eq:sup2} and \eqref{eq:sup3} for regions $\Rt_2$ and $\Rt_3$. Note that for any $(w,s)\not\in\ctsup$, we have $s<\st_2$. Therefore, the comparison principles can be used to show that the trajectory cannot cross the boundary. It is possible that the trajectory $(w(t),s(t))$ to move above $s=\st_2$, in which case the Lyapunov functions are no longer well-defined. However, comparison principles are not needed in this region, as $\R\times[\st_2,\infty)\subseteq\ctsup$.
\end{proof}

Next, we provide a more detailed description of the dynamics \eqref{S}. The trajectory moves through the regions in the order $\Rt_4, \Rt_3, \Rt_2$, and $\Rt_1$, ultimately reaching the  $w$-axis, which leads to a blowup of $\rho = 1/s$. Furthermore, the dynamics remain in each region for a finite amount of time. Consequently, the blowup occurs in finite time.

\begin{proposition}[Finite-time blowup]
  Consider the dynamics \eqref{S} with initial data $(w_0,s_0)\in \ctsup$. Then, there exists a finite time $T_*$, such that
	\[ w(T_*)\leq 0,\quad s(T_*)=0.\]	
\end{proposition}\label{prop:blowup}
\begin{proof}
 We divide the supercritical region $\ctsup$ into four areas $\{A_i\}_{i=1}^4$, defined as follows:
 \begin{align*}
 	A_1 & = \ctsup~\cap~\big\{(w,s): w>0,~ 0<s<\st_*\big\},\quad \st_*=\tfrac12(\st_1+\st_2),\\
 	A_2 & = \ctsup~\cap~\big\{(w,s): w>0,~ s\geq\st_*\big\},\\
	A_3 & = \ctsup~\cap~\big\{(w,s): w\leq 0,~ s>\st_{**}\big\},\quad \st_{**}=\tfrac12\st_3=\tfrac{1}{2\cM},\\
	A_4 & = \ctsup~\cap~\big\{(w,s): w\leq 0,~ 0<s\leq\st_{**}\big\}.
\end{align*}
Refer to Figure \ref{fig:blowup} for an illustration of these four areas. We will demonstrate that a trajectory progresses sequentially through $A_1$, $A_2$, $A_3$, and $A_4$, ultimately exiting along the negative $w$-axis. Moreover, the solution remains in each area for a finite amount of time.
\begin{figure}[h!]
    \includegraphics{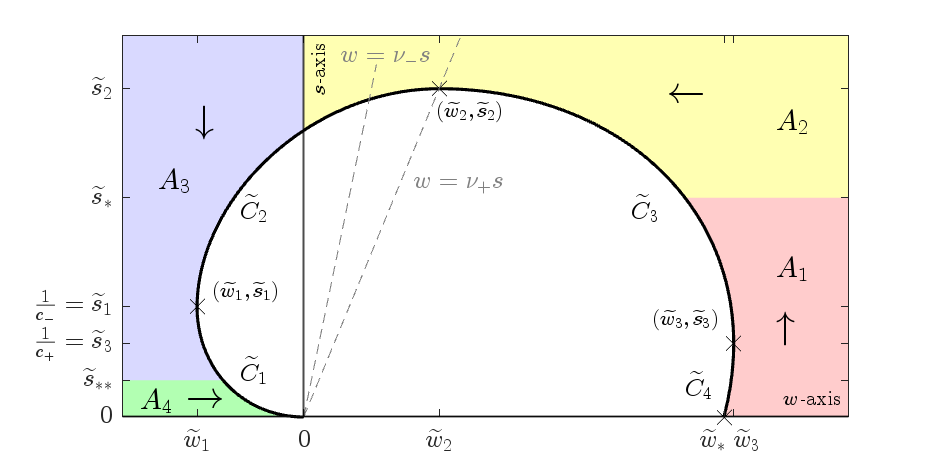}
\caption{Illustration of the flow in the supercritical region $\ctsup$}\label{fig:blowup}
\end{figure}

\noindent\textbf{Step 1:} In area $A_1$, we claim that the trajectory moves upward with a minimum speed $v_s>0$, namely
\begin{equation}\label{eq:A1}
	s'(t)\geq v_s>0, \quad\text{for any}~~(w(t),s(t))\in A_1.	
\end{equation}
By Proposition \ref{prop:sup}, the trajectory cannot cross $\Ct_3\cup\Ct_4$. Hence, it may only exit $A_1$ and enter $A_2$. Furthermore, this transition must occur within a finite time $t\leq \tfrac{\st_*}{v_s}$, since \eqref{eq:A1} implies
\[s(t)\geq s_0 + v_st\geq \st_*,\quad\text{if}~~ t\geq\tfrac{\st_*}{v_s}.\]

To prove \eqref{eq:A1}, we apply \eqref{S} and obtain
\[s'(t)= w(t) - \nu s(t)\geq w(t) - \nvM s(t).\]
The positive lower bound corresponds to the fact that $A_1$ is distance away from the line $w=\nvM s$, as illustrated in Figure \ref{fig:blowup}. To check this analytically, we split into two cases.

\textit{Case 1:} $(w(t),s(t))\in A_1\cap\Rt_3$. 
Applying comparison principle \eqref{eq:sup3}, we obtain 
 \[s'(t) \geq \Lt_3^{++}(w(t),s(t))+\sqrt{2\Nt_3^{++}(s(t))}\geq \sqrt{2\Nt_3^{++}(s(t))}>0,\]
 where the last strict inequality uses the property of $\Nt_3^{++}$. See Lemma \ref{lem:Nt3}.

\textit{Case 2:} $(w(t),s(t))\in A_1\cap\Rt_4$. 
Applying comparison principle \eqref{eq:sup4} and Lemma \ref{lem:Nt4}, we have 
\[s'(t) \geq w(t) - \nvM s(t) > \min\big\{\sqrt{\tfrac k\cM},\tfrac{\nvm}{2\cM}\big\}>0.\]
 
\noindent\textbf{Step 2:} In area $A_2$, we apply \eqref{S} and obtain that the trajectory moves leftward with
\begin{equation}\label{eq:A2}
	w'(t)=k(1-cs(t))\leq k(1-\cm \st_*)=-\tfrac{k}{2}(\cm \st_2-1) =: -v_w, \,\,\text{for any}~~(w(t),s(t))\in A_2,
\end{equation}
where the minimum speed $v_w=\tfrac{k}{2}(\cm \st_2-1)>0$ as $\st_2>\st_1=1/\cm$.
By Proposition \ref{prop:sup}, the trajectory cannot cross $\Ct_2\cup\Ct_3$. Hence, it may only exit $A_2$ and enter $A_3$. Furthermore, this transition must occur within a finite time $t\leq \tfrac{w_0}{v_w}$, since \eqref{eq:A2} implies
\[w(t)\leq w_0 - v_wt\leq 0,\quad\text{if}~~ t\geq\tfrac{w_0}{v_w}.\]

\noindent\textbf{Step 3:} In area $A_3$, we apply \eqref{S} and obtain that the trajectory moves downward with
\begin{equation}\label{eq:A3}
	s'(t)\leq w(t) - \nvm s(t)\leq 0 - \nvm \st_{**}=-\tfrac{\nvm}{2\cM}<0.
\end{equation}
By Proposition \ref{prop:sup}, the trajectory cannot cross $\Ct_1\cup\Ct_2$. Hence, it may only exit $A_3$ and enter $A_4$. Furthermore, this transition must occur within a finite time $t\leq \tfrac{2\cM s_0}{\nvm}$, since \eqref{eq:A3} implies
\[s(t)\leq s_0 - \tfrac{\nvm}{2\cM}t\leq 0,\quad\text{if}~~ t\geq\tfrac{2\cM s_0}{\nvm}.\]
 
\noindent\textbf{Step 4:} In area $A_4$, we apply \eqref{S} and obtain that the trajectory moves rightward with
\begin{equation}\label{eq:A4}
	w'(t)\geq k(1-\cM \st_{**})\leq 0 - \nvm \st_{**}=\tfrac{k}{2}>0.
\end{equation}
By Proposition \ref{prop:sup}, the trajectory cannot cross $\Ct_1$. Hence, it may only exit $A_4$ through the negative $w$-axis $\{(w,0): w\leq0\}$. Furthermore, this transition must occur within a finite time $t\leq \tfrac{-2w_0}{k}$, since \eqref{eq:A4} implies
\[w(t)\geq w_0 +\tfrac{k}{2}t\geq 0,\quad\text{if}~~ t\geq\tfrac{-2w_0}{k}.\] 

Combining the four steps concludes the proof.

For Scenario III, $\ctsup$ consists of $A_2$, $A_3$ and $A_4$. The proof follows the same argument. Since $\st_2$ is not defined, we may take $\st_*=\st_1+1$. Everything else follows similarly.
\end{proof}

\begin{proof}[Proof of Theorem \ref{thm:sup}]
	Suppose there exists $x_*\in\T$ such that $(w_0(x_*),s_0(x_*))\in\ctsup$. We consider the dynamics \eqref{S} along the characteristic path $X(t;x_*)$. Applying Proposition \ref{prop:blowup}, we obtain that there exists a finite time $T_*$ such that 
	\[w(X(T_*;x_*),T_*)<0,\quad \text{and}\quad s(X(T_*;x_*),T_*)=0.\]
	From the relations \eqref{eq:ws} and \eqref{eq:G}, we deduce
	\[\rho(X(T_*;x_*),T_*)=\infty,\quad G(X(T_*;x_*),T_*)=-\infty, \quad \text{and}\quad \pa_xu(X(T_*;x_*),T_*)=-\infty.\]
	Therefore, the solution developed a singular shock at time $T_*$ and location $X(T_*;x_*)$.
	
	The supercritical region $(w,s)\in\ctsup$ can be translated to $(G,\rho)\in\CTsup$ using the relation \eqref{eq:ws}. Thanks to the Lyapunov functions $\Lt^{\pm\pm}$, the region can be expressed explicitly in terms of $\Pt^{\pm\pm}$ and $\Nt^{\pm\pm}$, as stated in Theorem \ref{thm:sup}.
\end{proof}

\section{Auxiliary phase plane analysis}\label{sec:phase}
In this section, we perform phase plane analysis to the coupled dynamics \eqref{S}, and obtain explicit expressions on the quantities and conditions in our main critical threshold theorems.

Rewrite the localized system \eqref{S} in matrix form:
\[
    \begin{bmatrix}
        w-\frac{\nu}{c} \\[5pt]
        s-\frac{1}{c}
    \end{bmatrix}^\prime = \begin{bmatrix}
        0 & -kc \\
        1 & -\nu
    \end{bmatrix} \begin{bmatrix}
        w-\frac{\nu}{c} \\[5pt]
        s-\frac{1}{c}
    \end{bmatrix}.
\]
The matrix has two eigenvalues:
\begin{equation}\label{eq:lambda12}
\lambda_1 := \frac{-\nu+\sqrt{\nu^2-4kc}}{2}, \ \ \textnormal{and} \ \ \lambda_2 := \frac{-\nu-\sqrt{\nu^2-4kc}}{2}.	
\end{equation}
There are two scenarios:
\begin{itemize}
 \item[(I).] If $\nu\geq 2\sqrt{kc}$, then $\lambda_1$ and $\lambda_2$ are both real and negative. Solutions takes the explicit form:
 	\[w(t)=\frac{\nu}{c}-A_1kce^{\lambda_1t}-A_2kce^{\lambda_2t},
 	\quad
 	s(t)=\frac{1}{c}+A_1\lambda_1e^{\lambda_1t}+A_2\lambda_2e^{\lambda_2t},\]
 	where $A_1$ and $A_2$ are determined by initial data.
 	Note that when $\nu=2\sqrt{kc}$, the two eigenvalues $\lambda_1=\lambda_2$. Then $e^{\lambda_2t}$ should be replaced by $te^{\lambda_1t}$. We will omit the detailed calculation for this borderline case.
 \item[(II).] If $\nu<2\sqrt{kc}$, then $\lambda_1$ and $\lambda_2$ are complex-valued, with
 \[\lambda_1=-\frac{\nu}{2}+i\theta,\quad \lambda_2=-\frac{\nu}{2}-i\theta,\quad {where}\quad \theta=\frac12\sqrt{4kc-\nu^2}.\]
 Solution are \emph{oscillatory}, satisfying the following explicit form: 
 	\begin{align*}
	w(t) & =\frac{\nu}{c}-e^{-\frac{\nu t}{2}}\big(A_1kc\cos(\theta t)+A_2kc\sin(\theta t)\big),\\
 	s(t) & =\frac{1}{c}+e^{-\frac{\nu t}{2}}\Big((-A_1\tfrac{\nu}{2}+A_2\theta)\cos(\theta t)-(A_1\theta+A_2\tfrac{\nu}{2})\sin(\theta t)\Big),
 	\end{align*}
 	where $A_1$ and $A_2$ are determined by initial data.
\end{itemize}
Since the curves $\{C_i\}_{i=1}^4$ and $\{\Ct_i\}_{i=1}^4$ are trajectories of the localized systems (\ref{S}$^{\pm\pm}$), we will perform direct calculations to obtain explicit descriptions of the thresholds.

\subsection{Subcritical threshold}
We split the calculation into four steps, corresponding to the curves $C_1,\ldots,C_4$ sequentially.

\medskip\noindent\textbf{Step 1:}
In region $R_1=\{(w,s): 0<s<1/\cM, w<\nvM s\}$, constructing the trajectory $C_1$ amounts to solving (\ref{S}$^{++}$) backwards in time,  subject to the initial condition 
    \[
    w(0) = 0, \quad s(0) = 0.
    \]
    
There are two cases.
    
\medskip\noindent(I).
If $\nu_+ \geq 2\sqrt{kc_+}$, then the solution $(w(t),s(t))$ takes the following form:
\begin{align*}
    w(t) &=\frac{\nu_+}{c_+}-\frac{k}{\lambda_1^{++}-\lambda_2^{++}}\bigg(\frac{\lambda_2^{++}}{\lambda_1^{++}}\exp(\lambda_1^{++}t)-\frac{\lambda_1^{++}}{\lambda_2^{++}}\exp(\lambda_2^{++}t)\bigg),\\
    s(t) &= \frac{1}{c_+}+\frac{\lambda_2^{++}}{c_+(\lambda_1^{++}-\lambda_2^{++})}\exp(\lambda_1^{++}t)-\frac{\lambda_1^{++}}{c_+(\lambda_1^{++}-\lambda_2^{++})}\exp(\lambda_2^{++} t).
\end{align*}
Let $t_1<0$ be the time at which $C_1$ exits the region at $(w_1, s_1)$, namely
\[w(t_1)=w_1,\quad s(t_1) = s_1 = \tfrac{1}{c_+}.\] 
Using this relation, We solve for $t_1$ as:
\[ t_1 = \frac{1}{\lambda_1^{++}-\lambda_2^{++}}\log\bigg(\frac{\lambda_1^{++}}{\lambda_2^{++}}\bigg).\]
Consequently, we obtain the expression of $w_1$:
\[
     w_1  = \frac{\nu_+}{c_+} - \eta_1,\quad\text{where}\quad \eta_1:=\bigg(\frac{\lambda_2^{++}}{\lambda_1^{++}}\bigg)^{\frac{\nu_+}{2(\lambda_1^{++}-\lambda_2^{++})}}\sqrt{\frac{k}{c_+}}.
\]

\medskip\noindent(II).
 If $0< \nu_+ < 2\sqrt{kc_+}$, then the solution $(w(t),s(t))$ takes the following form:
      \begin{align*}
    w(t) &= \frac{\nvM}{\cM}+e^{-\frac{\nvM t}{2}}\bigg[-\frac{\nvM}{\cM}\cos(\theta^{++} t)+\left(\frac{k}{\theta^{++}}-\frac{\nvM^2}{2c_+\theta^{++}}\right)\sin(\theta^{++}t)\bigg],\\
    s(t) &= \frac{1}{\cM}+e^{-\frac{\nvM t}{2}}\bigg[-\frac{1}{\cM}\cos(\theta^{++}t)-\frac{\nvM}{2c_+\theta^{++}}\sin(\theta^{++}t)\bigg].
\end{align*}
In this case, $t_1$, the time at which $C_1$ exits the region, takes the following form:
\[
t_1 = -\frac{1}{\theta^{++}}\arctan\bigg(\frac{2\theta^{++}}{\nvM}\bigg).
\]
Plugging $t_1$ into the expression of $w(t)$, we obtain
\[
 w_1 = \frac{\nu_+}{c_+}- \sqrt{\frac{k}{c_+}}z_1,\quad z_1:=\exp\bigg(\frac{\nvM\arctan(\frac{2\theta^{++}}{\nvM})}{2\theta^{++}}\bigg).
\]

To summarize, we have the explicit expression of $w_1$ as follows:
\begin{equation}\label{eq:w1}
	w_1=\begin{cases}
	\frac{\nvM}{\cM} - \eta_1,&\nvM\geq2\sqrt{k\cM},\\
	\frac{\nvM}{\cM}- \sqrt{\frac{k}{\cM}}z_1, & 0< \nvM<2\sqrt{k\cM}.
\end{cases}
\end{equation}

\medskip\noindent\textbf{Step 2:} 
In region $R_2=\{(w,s):s>1/\cM, w<\nvm s\}$, we construct the trajectory $C_2$ by solving (\ref{S}$^{+-}$) backward in time, subject to the initial condition 
  \[
  w(t_1) = w_1,\quad s(t_1) = \tfrac{1}{c_+}.
  \]
  
There are two cases:

\medskip\noindent(I). In the case of \emph{strong alignment}, when $\nu_- \geq 2\sqrt{kc_+}$, the solution $(w(t),s(t))$ is given as
\begin{subequations}\label{eq:P2strong}
\begin{align}
    w(t) &= \frac{\nvm}{\cM}-\frac{\tfrac{\nvm}{\cM}-w_1}{\lambda_1^{+-}-\lambda_2^{+-}}\bigg(-\lambda_2^{+-}\exp\big(\lambda_1^{+-}(t-t_1)\big)+\lambda_1^{+-}\exp\big(\lambda_2^{+-}(t-t_1)\big)\bigg),\label{eq:wP2strong}\\
    s(t) &= \frac{1}{c_+}- \frac{\tfrac{\nvm}{\cM}-w_1}{\lambda_1^{+-}-\lambda_2^{+-}}\bigg(\exp\big(\lambda_1^{+-}(t-t_1)\big)-\exp\big(\lambda_2^{+-}(t-t_1)\big)\bigg).\label{eq:sP2strong}
\end{align}
\end{subequations}	
We apply these explicit expressions to verify that the trajectory $C_2$ won't exit $R_2$. Indeed, we have
\begin{equation}\label{eq:C2R2}
  w(t) - \nvm s(t) = - \frac{\tfrac{\nvm}{\cM}-w_1}{\lambda_1^{+-}-\lambda_2^{+-}}\bigg(\lambda_1^{+-}\exp\big(\lambda_1^{+-}(t-t_1)\big)-\lambda_2^{+-}\exp\big(\lambda_2^{+-}(t-t_1)\big)\bigg)<0,
\end{equation}
for any $t\leq t_1$. Correspondingly, we obtain the following Lemma on $P_2^{+-}$.
\begin{lemma}\label{lem:P2strong}
 The solution $P_2^{+-}$ of \eqref{aux_P_2} is well-defined in $[s_1,\infty)$. Moreover, we have
\[P_2^{+-}(s)>0,\quad\forall~s\in[s_1,\infty).\]
\end{lemma}
\begin{proof}
    The trajectory $C_2$ is given by the zero level set of $\mathcal{L}_2^{+-}$, namely
    \[
    w(t) - \nu_- s(t) = -\sqrt{2P_2^{+-}(s(t))}.  
    \]
    From \eqref{eq:C2R2}, we obtain
    \[P_2^{+-}(s(t))>0,\quad\forall~t\leq t_1.\]
    Moreover, we observe from \eqref{eq:sP2strong} that
    \[\text{Range}(\{s(t):t\leq t_1\}) = [s_1,\infty).\]
    This implies that $P_2^{+-}(s)$ is well-defined and is strictly positive for all $s\in [s_1,\infty)$.
\end{proof}

Moreover, we examine the asymptotic behavior of the trajectory $(w(t),s(t))$ as $t\to-\infty$. From the relation \eqref{eq:ws}, we have
\[G(t)=\tfrac{w(t)}{s(t)},\quad \rho(t)=\tfrac{1}{s(t)}.\]
As $t\to-\infty$, the trajectory in the $(G,\rho)$ plane converges to $(G_\flat,0)$, illustrated in Figure \ref{fig:CT}. We apply \eqref{eq:P2strong} and obtain:
 \begin{equation}\label{eq:Gflat}
  G_\flat = \lim_{t\to-\infty}\frac{w(t)}{s(t)} = -\lambda_1^{+-} = \frac{\nvm-\sqrt{\nvm^2-4k\cM}}{2}.
 \end{equation}
 
 The borderline case $\nvm=2\sqrt{k\cM}$ can be treated similarly, yielding the same results.
 
\medskip\noindent(II).
In the case of \emph{median and weak alignment}, when $0 < \nu_- < 2\sqrt{kc_+}$, the solution $(w(t),s(t))$ is given as
\begin{align*}
    w(t) &= \frac{\nu_-}{c_+}-e^{-\frac{\nvm (t-t_1)}{2}}\Big(\big(\tfrac{\nvm}{\cM}-w_1\big)\cos(\theta^{+-} (t-t_1))+\big(\tfrac{\nvm}{\cM}-w_1\big)\tfrac{\nu_-}{2\theta^{+-}}\sin(\theta^{+-} (t-t_1))\Big),\\
    s(t) &= \frac{1}{\cM}-e^{-\frac{\nvm (t-t_1)}{2}}\tfrac{1}{\theta^{+-}}\big(\tfrac{\nvm}{\cM}-w_1\big)\sin(\theta^{+-}(t-t_1)).
\end{align*}
In this case, let $t_2$ be the time at which $C_2$ exits the region at $(w_2, s_2) = (w(t_2), s(t_2))$ where $w_2 = \nu_- s_2$. From this relation we solve for $t_2$:
\[
t_2 = t_1 - \frac{1}{\theta^{+-}}\bigg(\pi-\arctan\Big(\frac{2\theta^{+-}}{\nvm}\Big)\bigg).
\]
Then $(w_2, s_2)$ are given explicitly by
\begin{align}
    w_2 & = \frac{\nvm}{\cM}+\frac{\nvm}{\sqrt{k\cM}}\big(\tfrac{\nvm}{\cM}-w_1\big)z_2,\quad z_2:=\exp\bigg(\frac{\nvm(\pi-\arctan\big(\frac{2\theta^{+-}}{\nvm})\big)}{2\theta^{+-}}\bigg),\label{eq:w2}\\
    s_2 & = \frac{1}{c_+}+\frac{1}{\sqrt{kc_+}}\big(\tfrac{\nvm}{\cM}-w_1\big)z_2.\label{eq:s2}
\end{align}
In addition, we have the following lemma used to complement the proof to strong comparison principle \eqref{eq:sub2}.
\begin{lemma}\label{lem:P2weak}
The solution $P_2^{+-}$ of \eqref{aux_P_2} is well-defined in $[s_1,s_2]$. Moreover, we have $P_2^{+-}(s_2)=0$, and
\[P_2^{+-}(s)>0,\quad\forall~s\in[s_1,s_2).\]
\end{lemma}
\begin{proof}
 We compute
 \begin{align*}
  \sqrt{2P_2^{+-}(s)} & = - w(t) + \nu_-s(t)\\
  &  = \big(\tfrac{\nvm}{\cM}-w_1\big)e^{-\frac{\nvm (t-t_1)}{2}} \Big(\cos\big(\theta^{+-}(t-t_1)\big)-\frac{\nvm}{2\theta^{+-}}\sin\big(\theta^{+-}(t-t_1)\big)\Big).
 \end{align*}
 Observe that $ \big(\tfrac{\nvm}{\cM}-w_1\big)e^{-\frac{\nvm (t-t_1)}{2}} > 0 $ and 
 \[
\cos(\theta^{+-}(t-t_1))-\frac{\nvm}{2\theta^{+-}}\sin(\theta^{+-}(t-t_1)) \geq 0, \quad \forall \ t \in [t_2, t_1],
\]  
with strict inequality holds if $t\in(t_2, t_1]$. Therefore, $P_2^{+-}$ is well-defined in $[s_1,s_2]$, and is strictly positive for $s\in[s_1,s_2)$.
\end{proof}

To make sure $(w_2,s_2)$ lie at the boundary of $R_3$, we require the first admissible condition \eqref{AC1}, namely $s_2>1/\cm$. Using \eqref{eq:s2}, we obtain an explicit expression of \eqref{AC1} as follows:
\begin{equation}\label{AC1e}\tag{AC1e}
 	\frac1\cm-\frac1\cM<\frac{1}{\sqrt{kc_+}}\big(\tfrac{\nvm}{\cM}-w_1\big)z_2.
\end{equation}
\begin{remark}\label{rem:AC1}
 The admissible condition \eqref{AC1e} holds when $\cM$ and $\cm$ are close enough. In particular, if $c$ is a constant, \eqref{AC1e} holds automatically.
 
 Moreover, when $\nu$ is a constant, the admissible condition \eqref{AC1e} coincides with the first closing condition in \cite[Theorem 1.9, Case 2.1]{choi2024critical}. To see this, we let $\nu=\nvM=\nvm$. Since $\nvm=\nvM<2\sqrt{k\cM}$, the condition \eqref{AC1e} reads:
 \[
  \frac1\cm-\frac1\cM<\frac{1}{\sqrt{kc_+}}\Big(\tfrac{\nu}{\cM}-\big(\tfrac{\nu}{\cM}- \sqrt{\tfrac{k}{\cM}}z_1\big)\Big)z_2 = \frac1\cM z_1z_2,
 \]
 or equivalently, $\frac{1+z_1z_2}{\cM}>\frac1\cm$. In \cite{choi2024critical}, the authors denote $s_+=\frac{1+z_1z_2}{\cM}$. Therefore, we arrive at their closing condition $s_+\cm>1$. 	
\end{remark}

\medskip\noindent\textbf{Step 3:} 
In region $R_3=\{(w,s):s>1/\cm, w>\nvm s\}$, we construct the trajectory $C_3$ by solving (\ref{S}$^{--}$) backward in time, subject to the initial condition
     \[
     w(t_2) = w_2,\quad s(t_2) = s_2.
     \]

There are two cases.

\medskip\noindent(I).
When $\nu_-\geq 2\sqrt{kc_-}$, the solution \((w(t),s(t))\) is given as
    \begin{align}
        w(t) & = \frac{\nu_-}{c_-}+\frac{(\lambda_2^{--})^2}{\lambda_1^{--}-\lambda_2^{--}}(s_2-\tfrac1\cm)e^{\lambda_1^{--}(t-t_2)}-\frac{(\lambda_1^{--})^2}{\lambda_1^{--}-\lambda_2^{--}}(s_2-\tfrac1\cm)e^{\lambda_2^{--}(t-t_2)} \label{eq:wN3medium},\\
        s(t) &= \frac{1}{c_-}-\frac{\lambda_2^{--}}{\lambda_1^{--}-\lambda_2^{--}}(s_2-\tfrac1\cm)e^{\lambda_1^{--}(t-t_2)}+\frac{\lambda_1^{--}}{\lambda_1^{--}-\lambda_2^{--}}(s_2-\tfrac1\cm)e^{\lambda_2^{--}(t-t_2)}. \label{eq:sN3medium}
    \end{align}
    Let $t_3$ be the time at which $C_3$ exits the region at $(w_3,s_3)$, namely
    \[
    w(t_3) = w_3, \quad s(t_3) = s_3 = \tfrac{1}{c_-}.
    \]
    Using this relation, we solve for $t_3$ and obtain
    \[
    t_3 = t_2+\frac{1}{\lambda_1^{--}-\lambda_2^{--}}\log\bigg(\frac{\lambda_1^{--}}{\lambda_2^{--}}\bigg). 
    \]
    This gives the explicit expression of $w_3$ as
    \[ w_3  = \frac{\nu_-}{c_-}+(s_2-\tfrac1\cm)\eta_3,\]
    where we use $\eta_3$ to denote
    \[\eta_3 := \frac{\lambda_1^{--}}{\lambda_2^{--}(\lambda_1^{--}-\lambda_2^{--})}\bigg((\lambda_1^{--})^2e^{\frac{\lambda_1^{--}}{\lambda_1^{--}-\lambda_2^{--}}}-(\lambda_2^{--})^2e^{\frac{\lambda_2^{--}}{\lambda_1^{--}-\lambda_2^{--}}}\bigg).
\]

\medskip\noindent(II).
When $\nu_-<\sqrt{2kc_-}$, the solution takes the following form:
\begin{align*}
    w(t) &= \frac{\nu_-}{c_-}+e^{-\frac{\nu_-(t-t_2)}{2}}\Big(\nu_-(s_2-\tfrac1\cm)\cos(\theta^{--}(t-t_2))\Big.\\
    &\hspace{1.5in}\Big.+\tfrac{\nu_-^2-4(\theta^{--})^2}{4\theta^{--}}(s_2-\tfrac1\cm)\sin(\theta^{--}(t-t_2))\Big), \\
    s(t) &= \frac{1}{c_-}+e^{-\frac{\nu_-(t-t_2)}{2}}\Big((s_2-\tfrac1\cm)\cos(\theta^{--}(t-t_2))+\frac{\nu_-}{2\theta^{--}}(s_2-\tfrac1\cm)\sin(\theta^{--}(t-t_2))\Big). 
\end{align*}
Let $t_3$ be the time at which $C_3$ exits the region at $(w_3,s_3)$. Solving for $t_3$, we obtain
\[
 t_3 = t_2-\frac{1}{\theta^{--}}\arctan\bigg(\frac{2\theta^{--}}{\nu_-}\bigg).
\]
Plugging $t_3$ into the solution to obtain $w_3$
\[
        w_3 = \frac\nvm\cm+\sqrt{k\cm}(s_2-\tfrac1\cm)z_3,\quad
        z_3:= \exp\bigg(\frac{\nvm\arctan\big({\frac{2\theta^{--}}{\nvm}}\big)}{2\theta^{--}}\bigg).
\]

To summarize, we have the explicit expression of $w_3$ as follows:
\begin{equation}\label{eq:w3}
	w_3=\begin{cases}
	\frac{\nu_-}{c_-}+(s_2-\tfrac1\cm)\eta_3,&\nu_- \geq 2\sqrt{kc_-},\\
	\frac\nvm\cm+\sqrt{k\cm}(s_2-\tfrac1\cm)z_3, & 0< \nu_- < 2\sqrt{kc_-}.
\end{cases}
\end{equation}

To make sure $(w_3,s_3)$ lie at the boundary of $R_4$, we require the second admissible condition \eqref{AC2}, namely $w_3>\nvM/\cm$. Using \eqref{eq:w3}, we obtain an explicit expression of \eqref{AC2} as follows:
\begin{equation}\label{AC2e}\tag{AC2e}
 	\nvM-\nvm < \cm (s_2-\tfrac1\cm)\times\begin{cases}
\eta_3,&\nu_- \geq 2\sqrt{kc_-},\\
	\sqrt{k\cm}z_3, & 0< \nu_- < 2\sqrt{kc_-}.	
\end{cases}
\end{equation}
\begin{remark}\label{rem:AC2}
The second admissible condition is \emph{more restrictive} than the first. Indeed, if \eqref{AC1e} is violated, then the right-hand side of \eqref{AC2e} becomes non-positive, rendering the inequality invalid.

Conversely, if \eqref{AC1e} holds, then \eqref{AC2e} is satisfied provided that $\nvM$ and $\nvm$ are sufficiently close. In particular, when $\nu$ is constant, condition \eqref{AC2e} holds automatically.
 
Moreover, when $c$ is constant, the admissible condition \eqref{AC2e} coincides with the one established in \cite{bhatnagar2023critical}. To see this, we let $c=\cM=\cm$. Since $\nvm<2\sqrt{k\cM}=2\sqrt{k\cm}$, the condition \eqref{AC2e} reads:
\begin{align*}
 \nvM-\nvm & < c \sqrt{kc}\,(s_2-\tfrac1c)z_3 = c\big(\tfrac{\nvm}{c}-w_1\big)z_2z_3\\
 & = -(\nvM-\nvm)z_2z_3 +cz_2z_3\times\begin{cases}
	\eta_1,&\nvM\geq2\sqrt{kc},\\
	\sqrt{\frac{k}{\cM}}z_1, & 0< \nvM<2\sqrt{kc}.
\end{cases}	
\end{align*}
which can be further simplified to
\[
\nvM-\nvm < \frac{cz_2z_3}{1+z_2z_3}\times\begin{cases}
	\eta_1,&\nvM\geq2\sqrt{kc},\\
	\sqrt{\frac{k}{\cM}}z_1, & 0< \nvM<2\sqrt{kc}.
\end{cases}	
\]
In the medium alignment case, the condition 
\[\nvM-\nvm < \frac{cz_2z_3}{1+z_2z_3}\eta_1\]
precisely matches the admissible conditions in \cite[Eq (2.5)]{bhatnagar2023critical}, when translated into their notation.
For the weak alignment case, a more restricted condition \eqref{AC3} is required. See Remark \ref{rem:AC3}.
\end{remark}


\medskip\noindent\textbf{Step 4:} 
In region $R_4=\{(w,s): 0<s<1/\cm, w>\nvM s\}$, we construct the trajectory $C_4$ by solving (\ref{S}$^{-+}$) backwards in time,  subject to the initial condition
 \[
 w(t_3) = w_3, \quad s(t_3) = s_3 = \tfrac1\cm.
 \]

\medskip\noindent(I).
In the case of \emph{median alignment}, when $\nu_+ \geq 2\sqrt{kc_-}$, the solution takes the following form:
     \begin{align}
    w(t) &= \frac{\nu_+}{c_-}+ \frac{w_3-\frac{\nu_+}{c_-}}{\lambda_1^{-+}-\lambda_2^{-+}}\Big(-\lambda_2^{-+}\exp(\lambda_1^{-+}(t-t_3))+\lambda_1^{-+}\exp(\lambda_2^{-+}(t-t_3)\Big), \label{eq:wN4strong}\\
    s(t) &= \frac{1}{c_-}-\frac{w_3-\frac{\nu_+}{c_-}}{\lambda_1^{-+}-\lambda_2^{-+}}\Big(-\exp(\lambda_1^{-+}(t-t_3))+\exp(\lambda_2^{-+}(t-t_3))\Big).\label{eq:sN4strong}
\end{align}
Explicit computation yields
    \begin{equation}\label{eq:N4R4}
w(t)-\nvM s(t) = \frac{w_3-\frac{\nu_+}{c_-}}{\lambda_1^{-+}-\lambda_2^{-+}}\bigg(\lambda_1^{-+}\exp(\lambda_1^{-+}(t-t_3))-\lambda_2^{-+}\exp(\lambda_2^{-+}(t-t_3))\bigg) >0,    	
    \end{equation}
    for any $t\leq t_3$. This leads to the following Lemma.
\begin{lemma}\label{lem:N4strong}
 The solution $N_4^{-+}$ of \eqref{aux_N_4} is well-defined in $(-\infty,s_3]$. Moreover, we have
\[N_4^{-+}(s)>0,\quad\forall~s\in(-\infty,s_3].\]
\end{lemma}
\begin{proof}
    The trajectory $C_4$ is given by the zero level set of $\mathcal{L}_4^{-+}$, namely
    \[
    w(t) - \nvM s(t) = \sqrt{2N_4^{-+}(s(t))}.  
    \]
    From \eqref{eq:N4R4}, we obtain
    \[N_4^{-+}(s)>0,\quad\forall~s\in(-\infty,s_3].\]
    Moreover, we observe from \eqref{eq:sN4strong} that
    \[\text{Range}(\{s(t):t\leq t_3\}) = (-\infty, s_3].\]
    This implies that $N_4^{-+}(s)$ is well-defined and is strictly positive for all $s\in (-\infty, s_3]$.
\end{proof}

Let $t_*$ be the time at which $C_4$ exits the region at $(w_*,0)$, namely
\[
w(t_*) = w_*, \quad s(t_*) = s_* = 0.
\]
Note that such $t_*$ exists as $0\in \text{Range}(\{s(t):t\leq t_3\})$.

Apply Lemma \ref{lem:N4strong} with $s=0$, we obtain
\[w_* = w(t_*) = \nvM s(t_*)+\sqrt{2N_4^{-+}(s(t_*))}=\sqrt{2N_4^{-+}(0)}>0.\]
Therefore, the admissible condition \eqref{AC3} holds automatically in this case.

\medskip\noindent(II).
In the case of \emph{weak alignment}, when $0< \nu_+ < 2\sqrt{kc_-}$, the solution then takes the following form
\begin{align}
    w(t) &= \frac{\nu_+}{c_-}+(w_3-\tfrac\nvM\cm)e^{-\frac{\nu_+(t-t_3)}{2}}\Big(\cos(\theta^{-+} (t-t_3))+\frac{\nu_+}{2\theta^{-+}}\sin(\theta^{-+} (t-t_3))\Big), \label{eq:wN4weak}\\
    s(t) &= \frac{1}{c_-}+\frac{1}{\theta^{-+}}(w_3-\tfrac\nvM\cm)e^{-\frac{\nu_+(t-t_3)}{2}}\sin(\theta^{-+}(t-t_3))\label{eq:sN4weak}.
\end{align}
In this case, the trajectory can cross $w = \nvM s$ at a time $t_4<t_3$, and $(w(t_4),s(t_4)) = (w_4,s_4)$ with $w_4 = \nvM s_4$. We compute
\begin{equation}\label{eq:N4R4weak}
w(t) - \nu_+s(t) 
      = (w_3-\tfrac\nvM\cm)e^{-\frac{\nu_+(t-t_3)}{2}}\Big(\cos(\theta^{-+}(t-t_3))-\frac{\nu_+}{2\theta^{-+}}\sin(\theta^{-+}(t-t_3))\Big),
\end{equation}
which reaches zero at:
\begin{equation}\label{time4}
t_4 = t_3 - \tfrac{1}{\theta^{-+}} \left( \pi - \arctan\left( \tfrac{2\theta^{-+}}{\nu_+} \right) \right).
\end{equation}
Then $(w_4, s_4)$ are given explicitly by
\begin{align}
    w_4 & = \frac{\nvM}{\cm}-\frac{\nvM}{\sqrt{k\cm}}\big(w_3-\tfrac{\nvM}{\cm}\big)z_4,\quad z_2:=\exp\bigg(\frac{\nvM(\pi-\arctan\big(\frac{2\theta^{-+}}{\nvM})\big)}{2\theta^{-+}}\bigg),\label{eq:w4}\\
    s_4 & = \frac{1}{\cm}-\frac{1}{\sqrt{k\cm}}\big(w_3-\tfrac{\nvM}{\cm}\big)z_4.\label{eq:s4}
\end{align}
In addition, we have the following Lemma.
\begin{lemma}\label{lem:N4weak}
    The solution $N_4^{-+}$ of \eqref{aux_N_4} is well-defined in $[s_4,s_3]$. Moreover, we have $N_4^{-+}(s_4) = 0$ and
\[
N_4^{-+}(s) > 0, \quad \forall \ s\in (s_4,s_3].
\]
\end{lemma}
\begin{proof}
  From the definition of $N_4^{-+}$ and \eqref{eq:N4R4weak}, we have
  \[
      \sqrt{2N_4^{-+}(s(t))} =(w_3-\tfrac\nvM\cm)e^{-\frac{\nu_+(t-t_3)}{2}}\Big(\cos(\theta^{-+}(t-t_3))-\frac{\nu_+}{2\theta^{-+}}\sin(\theta^{-+}(t-t_3))\Big)>0,
  \]
  for any $t\in(t_4,t_3]$, where we have used the fact that $(w_3 - \frac{\nu_+}{c_-})e^{-\frac{\nu_+(t-t_3)}{2}} > 0$.
  
  Therefore, $N_4^{-+}$ is well-defined in $[s_4,s_3]$ and strictly positive for any $s\in (s_4,s_3]$. Moreover, $s$ is monotone in $[t_4,t_3]$. Therefore, $N_4^{-+}$ is well-defined in $[s_4,s_3]$, and is strictly positive for $s\in(s_4,s_3]$.
\end{proof}

The following lemma provides an equivalent formulation of the third admissible condition \eqref{AC3}: there exists a time $t_*$ such that $s(t_*)=0$ and $w(t_*)=w_*\geq0$.

\begin{lemma}\label{lem:AC3}
    The admissible condition \eqref{AC3} holds if and only if 
    \begin{equation}\label{AC3p}\tag{AC3'}
     s_4 \leq 0.	
    \end{equation}
\end{lemma}
\begin{proof}
 If \eqref{AC3} holds, then $\sqrt{2N_4^{-+}(0)}=w_*\geq0$ is well-defined. From Lemma \ref{lem:N4weak}, we know $0\in[s_4,s_3]$, and therefore $s_4\leq0$.

 Now assume $s_4 \leq0$. Note that $s(t_3)>0$ and $s(t_4)\leq0$. Then by intermediate value theorem, there exists a time $t_*\in[t_4,t_3)$ such that $s(t_*)=0$. Moreover, we apply Lemma \ref{lem:N4weak} and obtain $w_* = \sqrt{2N_4^{-+}(0)}\geq0$.
\end{proof}

Using \eqref{eq:s4}, we obtain the following explicit expression of the condition \eqref{AC3p}:
\begin{equation}\label{AC3e}\tag{AC3e}
	\frac1{\sqrt{\cm}}\leq \frac1{\sqrt{k}}(w_3-\tfrac\nvM\cm)z_4.
\end{equation}
It is the \emph{most restrictive} condition among the three. Indeed, if \eqref{AC2e} is violated, the right-hand side of \eqref{AC3e} is non-positive, and the inequality cannot hold.

However, this condition is only required in the most restrictive setup, where $\nvM<2\sqrt{k\cm}$. Plugging in the explicit expressions of $w_3$, $w_2$ and $w_1$ would yield the following (completely explicit) condition:
\begin{equation}\label{eq:AC3}
(\nvM-\nvm)\frac1{\sqrt{k\cm}}\Big(\frac1\cm+\sqrt{\frac{\cm}{\cM^3}}z_2z_3\Big)z_4+\Big(\frac1\cm-\frac1\cM\Big)(z_3z_4+1)\leq\frac1\cM (z_1z_2z_3z_4-1).        
\end{equation}
Note that $z_i>1$. Therefore, the condition \eqref{eq:AC3} is satisfied when $\nvM$ and $\nvm$, as well as $\cM$ and $\cm$, are sufficiently close.

\begin{remark}\label{rem:AC3}
We examine the condition \eqref{eq:AC3} for two special cases. First, when $\nu$ is a constant, \eqref{eq:AC3} becomes
\[z_3z_4\big(\tfrac{1+z_1z_2}{\cM}\cm-1\big)\geq1,\quad \text{i.e.}\quad e^{\gamma_-}(s_+c_--1)\geq1,\]
which is precisely the closing condition in \cite[Theorem 1.9, Case 2.2]{choi2024critical}, using their notation $s_+=\frac{1+z_1z_2}{c_+}$ and $\gamma_-=z_3z_4$.

Next, when $c$ is constant, \eqref{eq:AC3} becomes
\[
  \nvM-\nvm \leq \sqrt{kc}\,\frac{z_1z_2z_3z_4-1}{(1+z_2z_3)z_4},
\]
which is precisely the admissible condition in \cite[Eq (2.4)]{bhatnagar2023critical}.

Hence, our results cover both works \cite{bhatnagar2023critical,choi2024critical} as special cases.
\end{remark}

\subsection{Supercritical threshold}
The construction of the supercritical threshold follows a similar strategy to that of the subcritical region; below, we present only the essential quantities and expressions.

\medskip\noindent\textbf{Step 1:} 
In region $\Rt_1=\{(w,s): 0<s<1/\cm, w<\nvm s\}$, we construct the trajectory $\Ct_1$ by solving (\ref{S}$^{--})$ backward in time, subject to the initial condition
\[
w(0) = 0, \quad s(0) = 0.
\]
Let $\tt_1$ be the time at which the trajectory $\Ct_1$ exits the region $\Rt_1$ at $(w(\tt_1),s(\tt_1)) = (\wt_1,\st_1)$. Then we have the explicit expressions of $\wt_1$ as follows:
\begin{equation}\label{eq:wt1}
	\wt_1=\begin{cases}
	\frac{\nu_-}{c_-} - \Big(\frac{\lambda_2^{--}}{\lambda_1^{--}}\Big)^{\frac{\nu_-}{2(\lambda_1^{--}-\lambda_2^{--})}}\sqrt{\frac{k}{c_-}},&\nvm\geq2\sqrt{k\cm},\\
	\frac{\nu_-}{c_-}- \sqrt{\frac{k}{c_-}}z_3, & 0< \nvm<2\sqrt{k\cm}.
\end{cases}
\end{equation}

\medskip\noindent\textbf{Step 2:} 
In region $\Rt_2=\{(w,s):s>1/\cm, w<\nvM s\}$, we construct the trajectory $\Ct_2$ by solving (\ref{S}$^{-+}$) backward in time, subject to the initial condition 
\[
w(\tt_1) = \wt_1, \quad s(\tt_1) = \st_1.
\]

\medskip\noindent(I).
In the cases of \emph{strong alignment} and \emph{median alignment}, when $\nu_+ \geq 2\sqrt{kc_-}$, the trajectory $\Ct_2$ won't exit $\Rt_2$. Moreover, we have the asymptotic behavior of the trajectory $(w(t), s(t))$ as $t \rightarrow -\infty$, and consequently the trajectory in the $(G,\rho)$ converges to $(G_\sharp,0)$, illustrated in Figure \ref{fig:CT}, where 
    \begin{equation}\label{eq:Gsharp}
    G_\sharp = -\lambda_1^{-+} = \frac{\nu_+-\sqrt{\nu_+^2-4kc_-}}{2}.
    \end{equation}
The borderline case $\nu_+ = 2\sqrt{kc_-}$ can be treated similarly, and we obtain the same results.

\medskip\noindent(II).
In the case of \emph{weak alignment}, when $\nu_+<2\sqrt{kc_-}$, let $\tt_2$ be the time at which the trajectory $\Ct_2$ exits the region at $(w(\tt_2),s(\tt_2))=(\wt_2,\st_2)$ where
\begin{align}
     \wt_2 & = \frac{\nvM}{\cm}+\frac{\nvM}{\sqrt{k\cm}}\Big(\frac{\nu_+}{c_-}-\wt_1\Big)z_4,\label{eq:wt2}\\
    \st_2 & = \frac{1}{c_-}+\frac{1}{\sqrt{kc_-}}\Big(\frac{\nu_+}{c_-}-\wt_1\Big)z_4.\label{eq:st2}
\end{align}
We utilize the explicit expression to verify the first admissible condition in \eqref{eq:ACsup}.
\begin{proposition}\label{prop:superAC1}
 In the case of weak alignment, the admissible condition $\st_2>\frac{1}{\cM}$ holds.	
\end{proposition}
\begin{proof}
 From \eqref{eq:wt1} we have $\wt_1<\frac\nvm\cm$. Then we obtain from \eqref{eq:st2} that
 \[\st_2>\frac1\cm + \frac{1}{\sqrt{kc_-}}\Big(\frac\nvM\cm-\frac\nvm\cm\Big)z_4\geq\frac1\cm\geq\frac1\cM.\]
\end{proof}

\medskip\noindent\textbf{Step 3:} In region $\Rt_3=\{(w,s):s>1/\cM, w>\nvM s\}$, we construct the trajectory $\Ct_3$ by solving (\ref{S}$^{++}$) backward in time, subject to the initial condition
\[
w(\tt_2) = \wt_2, \quad s(\tt_2) = \st_2.
\]
We observe that this step only applies to the case of weak alignment, when $\nu_+<2\sqrt{kc_+}$. The explicit solution is given by
\begin{align*}
    w(t) &= \frac{\nu_+}{c_+} + e^{-\frac{\nu_+(t - \tt_2)}{2}} \Big( 
        \nu_+\left(\st_2 - \tfrac{1}{c_+}\right) \cos\left(\theta^{++}(t - \tt_2)\right) \\
    &\hspace{1.5in} + \tfrac{\nu_+^2 - 4(\theta^{++})^2}{4\theta^{++}} \left(\st_2 - \tfrac{1}{c_+}\right) \sin\left(\theta^{++}(t - \tt_2)\right) 
    \Big), \\
    s(t) &= \frac{1}{c_+} + e^{-\frac{\nu_+(t - \tt_2)}{2}} \Big( 
        \left(\st_2 - \tfrac{1}{c_+}\right) \cos\left(\theta^{++}(t - \tt_2)\right) + 
        \frac{\nu_+}{2\theta^{++}} \left(\st_2 - \tfrac{1}{c_+}\right) \sin\left(\theta^{++}(t - \tt_2)\right) 
    \Big).
\end{align*}
Let $\tt_3$ be the time at which $\Ct_3$ exits the region at $(w(\tt_3),s(\tt_3)) = (\wt_3,\st_3)$. Then solving for $\tt_3$ from $s(\tt_3) = 1/c_+$ we obtain
\[
\tt_3 = \tt_2 - \frac{1}{\theta^{++}} \arctan\left( \frac{2\theta^{++}}{\nu_+} \right).
\]
Consequently, we have 
\begin{equation}\label{eq:wt3}
    \wt_3 = w(\tt_3)  = \frac{\nu_+}{c_+} + \sqrt{kc_+}\left(\st_2 - \tfrac{1}{c_+}\right) z_1.
\end{equation}
Therefore, the second  admissible condition in \eqref{eq:ACsup} holds.

\begin{proposition}\label{prop:superAC2}
 The admissible condition $\wt_3>\frac{\nvm}{\cM}$ holds.	
\end{proposition}
\begin{proof}
 From \eqref{eq:st2} we get $\st_2>\frac{1}{\cm}\geq\frac{1}{\cM}$. Then
 we obtain from \eqref{eq:wt3} that
 $\wt_3 >\frac{\nvM}{\cM}\geq\frac{\nvm}{\cM}.$
\end{proof}

highlighted in the construction of subcritical region is automatically satisfied
\[
\frac{\nu_+}{c_+} + \sqrt{kc_+}\left(\st_2 - \tfrac{1}{c_+}\right) z_1 > \frac{\nu_+}{c_+}.
\]
In addition, we have the following lemma used to complement the proof of Proposition \ref{prop:blowup}.
\begin{lemma}\label{lem:Nt3}
The solution $\Nt_3^{++}$ is well-defined in $[\st_3,\st_2]$. Moreover, we have $\Nt_3^{++}(\st_2)=0$, and
\[\Nt_3^{++}(s)>0,\quad\forall~s\in[\st_3,\st_2).\]
\end{lemma}
\begin{proof}
Apply \eqref{eq:wt2} and \eqref{eq:st2} to compute
 \[
  \sqrt{2\Nt_3^{++}(s)} = w(t) - \nu_+s(t) = e^{-\frac{\nu_+(t - \tt_2)}{2}} \,\frac{\nu_+^2 + 4(\theta^{++})^2}{4\theta^{++}}\,  \big(\st_2 - \tfrac{1}{c_+} \big) \sin\left(-\theta^{++}(t - \tt_2)\right)>0,
 \]
for all $t \in [\tt_3, \tt_2)$. Therefore, $\Nt_3^{++}$ is well-defined and strictly positive on the corresponding interval $[\st_3, \st_2)$.
\end{proof}

\medskip\noindent\textbf{Step 4:} In region $\Rt_4=\{(w,s): 0<s<1/\cM, w>\nvm s\}$, we construct the trajectory $\Ct_4$ by solving (\ref{S}$^{+-}$) backward in time, subject to the initial condition
\[
w(\tt_3) = \wt_3, \quad s(\tt_3) = \st_3 = 1/c_-.
\]
Since $\nvm<2\sqrt{k\cM}$, the solution takes the following explicit form:
\begin{align}
    w(t) &= \frac{\nu_-}{c_+}+\big(\wt_3-\frac{\nu_-}{c_+}\big)e^{-\frac{\nu_-(t-\tt_3)}{2}}\bigg[\cos(\theta^{+-} (t-\tt_3))+\frac{\nu_-}{2\theta^{+-}}\sin(\theta^{+-} (t-\tt_3))\bigg], \label{eq:wt4}\\
    s(t) &= \frac{1}{c_+}+\frac{1}{\theta^{+-}}\big(\wt_3-\frac{\nu_-}{c_+}\big)e^{-\frac{\nu_-(t-\tt_3)}{2}}\sin(\theta^{+-}(t-\tt_3))\label{eq:st4}.
\end{align}
We claim that the third admissible condition in \eqref{eq:ACsup} holds.
\begin{proposition}\label{prop:superAC3}
	There exists a time $\tt_*<\tt_3$ such that $s(\tt_*)=0$ and $\wt_*=w(\tt_*)>0$.
\end{proposition}
\begin{proof}
 Apply \eqref{eq:wt1}, \eqref{eq:st2}, \eqref{eq:wt3} and utilize the fact that $z_i>1$ to obtain
 \begin{align*}
&\wt_1< \frac{\nvm}{\cm}-\sqrt{\frac{k}{\cm}},\quad \st_2> \frac{1}{\cm}+\frac{1}{\sqrt{k\cm}}\Big(\frac{\nvM-\nvm}{\cm}+\sqrt{\frac{k}{\cm}}\Big)\geq \frac2\cm,\quad\text{and}\\
&\wt_3> \frac\nvM\cM + \sqrt{k\cM}\Big(\frac2\cm-\frac1\cM\Big)\geq \frac\nvM\cM + \sqrt{\frac k\cM}.
 \end{align*}
Then, for any $t\in[\tt_3-\frac{\pi}{2\theta^{+-}},\tt_3)$, we deduce from \eqref{eq:st4} and get
\[
s(t)< \frac1\cM + \frac1{\sqrt{k\cM}}\cdot\Big(\frac{\nvM-\nvm}{\cM}+\sqrt{\frac k\cM}\Big)\cdot 1 \cdot \sin\big(\theta^{+-}(t-\tt_3)\big)\leq \frac{1}{\cM}\Big(1+\sin\big(\theta^{+-}(t-\tt_3)\big)\Big).
\]
In particular, we have $s(\tt_3-\frac{\pi}{2\theta^{+-}})<0$. By continuity of $s$, we conclude that there exists a $\tt_*\in (\tt_3-\frac{\pi}{2\theta^{+-}},\tt_3)$ such that $s(\tt_*)=0$.

Moreover, we estimate from \eqref{eq:wt4} and obtain the bound
\begin{align*}
 \wt_* & = w(\tt_*) = \frac{\nu_-}{c_+}+\big(\wt_3-\frac{\nu_-}{c_+}\big)e^{-\frac{\nu_-(t-\tt_3)}{2}}\cos\big(\theta^{+-} (\tt_*-\tt_3)\big)+\frac{\nu_-}{2}\big(s(\tt_*)-\frac1\cM\big)\\
 & > \frac{\nvm}{\cM} - \frac{\nu_-}{2\cM} = \frac{\nu_-}{2\cM} >0.
\end{align*}
This finishes the proof.
\end{proof}

Finally, we establish the following refined estimate, showing that the trajectory $\Ct_4$ remains uniformly separated from the line $w = \nvM s$, thereby supporting the proof of Proposition \ref{prop:blowup}.
\begin{lemma}\label{lem:Nt4}
 For any $t\in[\tt_*,\tt_3]$, we have the uniform bound
 \[
 w(t)-\nvM s(t) > \min\big\{\sqrt{\tfrac k\cM},\tfrac{\nvm}{2\cM}\big\}.
 \]
\end{lemma}
\begin{proof}
 Recall that we can parametrize $\Ct_4$ by 
 \[\Ct_4=\Big\{(w,s): w = \textsf{w}(s),\,\, s\in[0,\st_3]\Big\},
 \quad\text{where}\quad \textsf{w}(s) = \nvm s+\sqrt{2\Nt_4^{+-}(s)}.\]
 We show that $\textsf{w}(s)$ is concave for $s\in[0,\st_3]$. Indeed, we apply \eqref{eq:Nt4} and compute:
 \begin{align*}
  \textsf{w}''(s) & = \frac{\textrm{d}^2}{\textrm{d}s^2}\sqrt{2\Nt_4^{+-}(s)} = \frac{\textrm{d}}{\textrm{d}s}\Big(\frac{-\nvm\sqrt{2\Nt_4^{+-}(s)}+k(1-\cM s )}{\sqrt{2\Nt_4^{+-}(s)}}\Big) = \frac{k}{\sqrt2}\frac{\textrm{d}}{\textrm{d}s}\Big(\frac{1-\cM s}{\sqrt{\Nt_4^{+-}(s)}}\Big)\\
  & = \frac{k}{2\sqrt2 \big(\Nt_4^{+-}(s)\big)^{\frac32}}\Big(-2\cM\Nt_4^{+-}(s)-k(1-\cM s)^2+\nvm(1-\cM s)\sqrt{2\Nt_4^{+-}(s)}\Big).
 \end{align*}
 The right-hand side is negative due to Cauchy-Schwarz inequality:
 \[ 2\cM\Nt_4^{+-}(s) + k(1-\cM s)^2 \geq 2\sqrt{2k\cM\Nt_4^{+-}(s)(1-\cM s)^2}>\nvm(1-\cM s)\sqrt{2\Nt_4^{+-}(s)}, \]
 where we have also used and $\nvm < 2\sqrt{k\cM}$.
 
 Now, consider the function $f(s) = \mathsf{w}(s) -\nvM s$. Note that $w(t)-\nvM s(t)=f(s(t))$. From Proposition \ref{prop:superAC3} we obtain
 \[
f(\st_3)  = \mathsf{w}(\st_3)-\nvM\st_3 = \wt_3-\frac\nvM\cm>\sqrt{\tfrac k\cM} \quad\text{and}\quad
f(0)  = \mathsf{w}(0) = w_*>\frac{\nvm}{2\cM}.	
 \]
 Since $f$ is concave in $[0,\st_3]$, we conclude that 
 $f(s)>\min\big\{\sqrt{\tfrac k\cM},\tfrac{\nvm}{2\cM}\big\}$. 	
\end{proof}

\section{Extensions}\label{sec:extension}
Several assumptions made in this paper can be relaxed. In this section, we briefly discuss possible extensions and generalizations of our main results.
\subsection{Spatial domain}\label{sec:ext:R}
We assume a periodic spatial domain $x\in\T$ throughout the paper. However, our results can be extended to the whole space $\R$. The critical thresholds established in Theorem \ref{thm:thresholds} are derived along each characteristic path, making them valid in the whole space case without modification. The primary technical challenge lies in the local well-posedness result of Theorem \ref{thm:LWP}, where additional care is needed to handle behavior at infinity. For instance, the estimate \eqref{eq:uinf} can not be obtained. In \cite{choi2024critical}, an extra assumption on $\rho_0 - c_0 \in \dot{H}^{-1}(\R)$ is introduced to control the behavior at infinity and ensure the propagation of regularity. This framework can be adapted to our general EPA system.

\subsection{Communication protocol}
Our result extends to the case where the communication protocol $\psi$ is weakly singular. In the absence of assumption \eqref{A5}, the bounds on $\nu$ in \eqref{eq:nuc} must be derived differently, by
\[\nvm=\|\psi\|_{L^1}\rho_{\min}(t)\leq\nu(x,t)=\int_\T\psi(x-y)\rho(x,t)\,\dd x\leq \|\psi\|_{L^1}\rho_{\max}(t)=\nvM,\] 
 where the bounds $\nu_\pm$ depending on the solution $\rho$. Following the approach in \cite{bhatnagar2023critical}, we may construct a smaller invariant region contained in $\R \times [\rho_-, \rho_+]$. For a detailed discussion on the optimal choice of $\rho_\pm$, see \cite{bhatnagar2023critical}.

\subsection{Presence of vacuum}
We assume $\rho_0 > 0$ in \eqref{eq:reginit} to ensure that the variables $(w, s)$ are well-defined. However, the presence of vacuum is also allowed. When $\rho_0(x) = 0$, rather than using the dynamics in \eqref{S}, we work directly with the equation \eqref{eq:rhoG}. Along the characteristic path, $\rho = 0$, and the dynamics reduce to
\begin{equation}\label{eq:vacuumG}	
 G'=-G(G-\nu)-kc.
\end{equation}
The following result, consistent with Figure \ref{fig:CT}, describes this scenario. 
\begin{proposition}Consider the dynamics \eqref{eq:vacuumG} with initial condition $G(0)=G_0$.
 \begin{itemize}
 	\item For weak alignment, $G(t)\to-\infty$ in finite time for any $G_0\in\R$.
 	\item For strong or median alignment, $G(t)\to-\infty$ in finite time if $G_0<G_\sharp$.
 	\item For strong alignment, $G(t)$ is bounded in all time if $G_0\geq G_\flat$.
 \end{itemize}
\end{proposition}
\begin{proof}
For weak alignment $\nvM<2\sqrt{k\cm}$, we deduce from \eqref{eq:vacuumG} and get
 \[G'=-\left(G-\frac\nu2\right)^2+\frac{\nu^2}{4}-kc\leq -\left(G-\frac\nu2\right)^2+\frac{\nvM^2}4-k\cm<-\left(G-\frac\nu2\right)^2,\]
leading to finite time blowup, for any initial data.

For median and strong alignment $\nvM\geq2\sqrt{k\cm}$, we have
\[G'(t)\leq -G(t)^2+\nvM G(t)-k\cm=-(G(t)-G_-)(G(t)-G_+),\quad \text{if}\,\, G(t)\geq0,\]
where
\[G_\pm=\frac{\nvM\pm \sqrt{\nvM^2-4k\cm}}{2},\quad\text{and in particular}\,\, G_-=G_\sharp.\]
Therefore, if $G_0<G_\sharp$, $G(t)$ will reach 0 in finite time. When $G(t)\leq0$, we have
\[G(t)'\leq -G(t)^2+\nvm G(t)-k\cm\leq -G(t)(G(t)-\nvm),\]
leading to finite time blowup.

For strong alignment $\nvm>2\sqrt{k\cM}$,  we deduce from \eqref{eq:vacuumG}
\[G'(t)\geq -G(t)^2+\nvm G(t)-k\cM=-(G(t)-\widetilde{G}_-)(G(t)-\widetilde{G}_+),\quad\text{if}\,\, G(t)\geq0,\]
where
\[\widetilde{G}_\pm=\frac{\nvm \pm \sqrt{\nvm^2-4k\cM}}{2},\quad\text{and in particular}\,\, \widetilde{G}_-=G_\flat.\]
Therefore, if $G_0\geq G_\flat$, we obtain the lower bound $G(t)\geq G_\flat$ for any $t\geq0$. On the other hand, $G'(t)\leq0$ when $G(t)\geq\widetilde{G}_+$. Hence, $G(t)$ is bounded above by $\max\{G_0, \widetilde{G}_+\}$.
\end{proof}

\subsection{Attractive Poisson force}
The EPA system with attractive Poisson force ($k < 0$) can be analyzed within the same framework. See related discussions in \cite{bhatnagar2020critical,choi2024critical}. In this case, the eigenvalues in \eqref{eq:lambda12} are always real, implying that the solutions do not exhibit oscillatory behavior, which generally simplifies the analysis. However, as noted in \cite{choi2024critical}, the resulting subcritical region may collapse to only include trivial initial data, suggesting that finite-time blowup is unavoidable in most cases.

\subsection{Multi-dimensions}
The critical threshold analysis of the multi-dimensional EPA system presents significant challenges, primarily due to the lack of control over the \emph{spectral gap} of $\nabla u$, a difficulty also encountered in the analysis of the multi-dimensional Euler-Poisson equations. To circumvent this issue, a Restricted Euler-Poisson (REP) system was introduced and studied in \cite{tadmor2003critical}. Similar ideas can be extended to formulate and analyze a restricted version of the EPA system.

\appendix

\section{Local well-posedness and regularity criterion}\label{sec:LWP}
In this section, we present the proof of Theorem \ref{thm:LWP}. We denote the fractional differential operator by $\Lambda^s := (-\Delta)^{s/2}$, and the homogeneous semi-norm by
\[\|f\|_{\dot{H}^s}=\|\Lambda^s f\|_{L^2}.\]
We also use the notation $\lesssim$, where $A \lesssim B$ means there exists a constant $C$, depending on parameters and initial data, such that $A\leq CB$.

\begin{proof}[Proof of Theorem \ref{thm:LWP}]
We perform a priori energy estimate on $(\theta:=\rho-c,G)$ to control the growth of the solution and establish global bounds.
    
First, we estimate $\|\theta\|_{\dot{H}^s}$ for any $s\geq0$. From the continuity equation \eqref{eq:main}$_1$, we obtain the dynamics of $\theta=\rho-c$ as:
\[
    \pa_t\theta+\pa_x(\theta u) = -\pa_tc-\pa_x(cu).
\]
Compute    
\begin{align*}
    &\frac{1}{2}\frac{\dd}{\dd t}\|\theta\|_{\dot{H}^s}^2 
    = \int \Lambda^s\theta\cdot \Lambda^s\pa_t\theta\,\dd x
    = -\int \Lambda^s\theta\cdot\Lambda^s\pa_x(\theta u)\,\dd x -\int \Lambda^s\theta\cdot\Lambda^s\big(\pa_tc+\pa_x(cu)\big)\,\dd x\\
    & \quad = -\int\Lambda^s\theta\cdot u\cdot \Lambda^s\pa_x\theta\,\dd x-\int \Lambda^s\theta\cdot [\Lambda^s\pa_x, u]\theta\,\dd x-\int \Lambda^s\theta\cdot\Lambda^s\big(\pa_tc+\pa_x(cu)\big)\,\dd x\\
    & \quad = \mathrm{I}_1+\mathrm{I}_2+\mathrm{I}_3,
\end{align*}
where $[\cdot,\cdot]$ denotes the commutator, whose operation is given as
\[[\mathcal{J},f]g = \mathcal{J}(fg)-f\mathcal{J}g.\]
Then \[\Lambda^s\pa_x((\rho-c)u) = [\Lambda^s\pa_x,u](\rho-c)+ \Lambda^s\pa_x(\rho-c)\cdot u\]
We estimate the three terms one by one. For $\mathrm{I}_1$, we have
\[\mathrm{I}_1=-\frac12\int \pa_x(\Lambda^s\theta)^2\cdot u\,\dd x=\frac12\int (\Lambda^s\theta)^2\cdot\pa_xu\,\dd x\leq\frac12 \|\pa_xu\|_{L^\infty}\|\theta\|_{\dot{H}^s}^2.\]
For $\mathrm{I}_2$, we apply the standard commutator estimates from \cite{kato1988commutator} and obtain
\begin{align*}
\mathrm{I}_2 &\leq \|\theta\|_{\dot{H}^s}\|[\Lambda^s\pa_x, u]\theta\|_{L^2}
\lesssim \|\theta\|_{\dot{H}^s}\big(\|\pa_xu\|_{L^\infty}\|\theta\|_{\dot{H}^s}+\|\theta\|_{L^\infty}\|\pa_xu\|_{\dot{H}^s}\big)\\
& \lesssim \|\pa_xu\|_{L^\infty}\|\theta\|_{\dot{H}^s}^2+(\|\rho\|_{L^\infty}+\cM)(\|G\|_{\dot{H}^s}+1+\|\theta\|_{\dot{H}^s})\|\theta\|_{\dot{H}^s}\\
& \lesssim (1+\|\rho\|_{L^\infty}+\|\pa_xu\|_{L^\infty})(1+\|\theta\|_{\dot{H}^s}^2+\|G\|_{\dot{H}^s}^2),
\end{align*}
where the relation \eqref{eq:G} is used in the penultimate inequality, along with the estimate
\[
\|\psi\ast\rho\|_{\dot{H}^s}\leq\|\psi\|_{L^1}\|\rho\|_{\dot{H}^s}\leq\psiM|\T|\big(\|\theta\|_{\dot{H}^s}+\|c\|_{\dot{H}^s}\big)\lesssim 1+\|\theta\|_{\dot{H}^s}.
\] 
Note that $\|\psi\|_{L^1}\leq\psiM|\T|$ is bounded, and hence can be absorbed into the constant. 

For $\mathrm{I}_3$, we use  the fractional Leibniz rule to get
\[
 \mathrm{I}_3 \lesssim \|\theta\|_{\dot{H}^s}\big(\|\pa_tc\|_{\dot{H}^s}+\|u\|_{L^\infty}\|c\|_{\dot{H}^{s+1}}+\|c\|_{L^\infty}\|\pa_xu\|_{\dot{H}^s}\big).
\]
We control $\|u\|_{L^\infty}$ as follows. Suppose $u$ is differentiable, then
\[|u(x)| = \left|u(x_*)+\int_{x_*}^x u'(y)\,\dd y\right|\leq |u(x_*)|+|\T|\cdot \|\pa_xu\|_{L^\infty}.\]
We may choose any $x_*$ such that $u(x_*)$ is below average, namely
\[u(x_*)\leq \bar{u}:=\frac{\int\rho(x)u(x)\,\dd x}{\int\rho(x)\,\dd x}\]
From conservation of mass and momentum, we have
\[|u(x_*,t)|\leq\frac{\left|\int\rho(x,t)u(x,t)\,\dd x\right|}{\int\rho(x,t)\,\dd x}=\frac{\left|\int\rho_0(x)u_0(x)\,\dd x\right|}{\int\rho_0(x)\,\dd x}\leq \|u_0\|_{L^\infty}.\]
Hence,
\begin{equation}\label{eq:uinf}
 \|u\|_{L^\infty}\leq \|u_0\|_{L^\infty}+|\T|\cdot\|\pa_xu\|_{L^\infty}\lesssim 1+\|\pa_xu\|_{L^\infty}.	
\end{equation}
We would like to point out that the estimate \eqref{eq:uinf} may not hold in the whole space $\Omega=\R$. See Section \ref{sec:ext:R} for more discussion.
Applying \eqref{eq:uinf} and the regularity assumption \eqref{A2} on $c$, we conclude with
\[
 \mathrm{I}_3 \lesssim (1+\|\pa_xu\|_{L^\infty})(1+\|\theta\|_{\dot{H}^s}^2).
\]
Combining the estimates on $\mathrm{I}_1$, $\mathrm{I}_2$ and $\mathrm{I}_3$, we end up with
\begin{equation}\label{eq:energytheta}
 \frac{\dd}{\dd t}\|\theta\|_{\dot{H}^s}^2 \lesssim (1+\|\rho\|_{L^\infty}+\|\pa_xu\|_{L^\infty})(1+\|\theta\|_{\dot{H}^s}^2+\|G\|_{\dot{H}^s}^2).	
\end{equation}

Next, we estimate $\|G\|_{\dot{H}^s}$, for any $s\geq0$. From the dynamics \eqref{eq:main}$_2$, we compute
\begin{align*}
    \frac{1}{2}\frac{\dd}{\dd t}\|G\|_{\dot{H}^s}^2 &= 
    \int \Lambda^sG \cdot \Lambda^s\pa_t G\,\dd x = -\int \Lambda^sG\cdot\Lambda^s\pa_x(G u)\,\dd x +k\int \Lambda^sG\cdot\Lambda^s\theta\,\dd x\\
    &= -\int\Lambda^sG\cdot u\cdot \Lambda^s\pa_xG\,\dd x-\int \Lambda^sG\cdot [\Lambda^s\pa_x, u]G\,\dd x+k\int \Lambda^sG\cdot\Lambda^s\theta\,\dd x\\
    &= \mathrm{II}_1+\mathrm{II}_2+\mathrm{II}_3.
\end{align*}
The terms $\mathrm{II}_1$ and $\mathrm{II}_2$ can be estimated using the similar argument as for $\mathrm{I}_1$ and $\mathrm{I}_2$, respectively.
\[
\mathrm{II}_1 \leq\frac12 \|\pa_xu\|_{L^\infty}\|G\|_{\dot{H}^s}^2,\quad\text{and}\quad
\mathrm{II}_2 \lesssim (1+\|\pa_xu\|_{L^\infty}+\|G\|_{L^\infty})(1+\|\theta\|_{\dot{H}^s}^2+\|G\|_{\dot{H}^s}^2).
\]
The term $\mathrm{II}_3$ can be bounded by
\[
\mathrm{II}_3 \lesssim \|\theta\|_{\dot{H}^s}^2+\|G\|_{\dot{H}^s}^2.
\]
Putting the estimates on $\mathrm{II}_1$, $\mathrm{II}_2$ and $\mathrm{II}_3$ together yields the bound:
\begin{equation}\label{eq:energyG}
  \frac{\dd}{\dd t}\|G\|_{\dot{H}^s}^2 \lesssim (1+\|\pa_xu\|_{L^\infty}+\|G\|_{L^\infty})(1+\|\theta\|_{\dot{H}^s}^2+\|G\|_{\dot{H}^s}^2).	
\end{equation}

where we have used the relation \eqref{eq:G} to estimate
\[
 \|G\|_{L^\infty}\leq \|\pa_xu\|_{L^\infty}+\|\psi\ast\rho\|_{L^\infty}\leq \|\pa_xu\|_{L^\infty} + \psiM\cdot\|\rho_0\|_{L^1}\lesssim 1+\|\pa_xu\|_{L^\infty}.
\]

Finally, we define the energy
\[ E_s(t) := \|\theta\|_{H^s}^2+\|G\|_{H^s}^2.\]
Combining the two energy estimates \eqref{eq:energytheta} and \eqref{eq:energyG} would yield
\begin{align}
\frac{\dd}{\dd t}E_s(t) &\lesssim (1+\|\rho(\cdot,t)\|_{L^\infty}+\|\pa_xu(\cdot,t)\|_{L^\infty}+\|G(\cdot,t)\|_{L^\infty})(1+E_s(t))\nonumber\\
&\lesssim (1+\|\rho(\cdot,t)\|_{L^\infty}+\|G(\cdot,t)\|_{L^\infty})(1+E_s(t)),\label{eq:Es}
\end{align}
for any $s\geq0$, where we have used the relation \eqref{eq:G} to estimate
\[
 \|\pa_xu\|_{L^\infty}\leq \|G\|_{L^\infty}+\|\psi\ast\rho\|_{L^\infty}\leq \|G\|_{L^\infty} + \psiM\cdot\|\rho_0\|_{L^1}\lesssim 1+\|G\|_{L^\infty}.
\]
If $s>\tfrac12$, we apply Sobolev embedding $H^s(\T)\hookrightarrow L^\infty(\T)$ and obtain
\[\frac{\dd}{\dd t}E_s(t)\lesssim (1+E_s(t))^{\frac32},\]
which implies local regularity \eqref{eq:regrhoG}. Consequently, \eqref{eq:regrhou} holds as
\[\|\pa_xu\|_{H^s}\leq \|G\|_{H^s}+\|\psi\ast\rho\|_{H^s}\lesssim 1+\|G\|_{H^s}+\|\theta\|_{H^s}<\infty,\]
and 
\[\|u\|_{L^2}\leq\|u\|_{L^\infty}\cdot|\T|^{-1/2}\lesssim 1+\|\pa_xu\|_{L^\infty}\lesssim 1+\|G\|_{L^\infty}<\infty.\]
Moreover, we apply Gr\"onwall inequality to \eqref{eq:Es} and get
\[E_s(T)\leq E_s(0)\exp\left(CT+C\int_0^T\left(\|\rho(\cdot,t)\|_{L^\infty}+\|G(\cdot,t)\|_{L^\infty}\right)\,\dd t\right).\]
Hence, $E_s(T)<\infty$ as long as criterion \eqref{eq:BKM} holds. 
\end{proof}

\bibliographystyle{plain}
\bibliography{bib_EPA}

\end{document}